\documentclass[a4paper]{article}

\usepackage{amsmath}
\usepackage{amsthm}
\usepackage{amssymb}
\usepackage{verbatim}
\usepackage{ifthen}
\usepackage{amsfonts}
\usepackage[mathscr]{euscript}   
\usepackage{xypic}      
\usepackage{graphicx}
\usepackage{textcomp}
\usepackage{yhmath}
\usepackage{bbm}
\usepackage{empheq}
\usepackage[bookmarks=true, linkbordercolor={0.4 0.35 0.8}]{hyperref}

\newtheorem{theorem}{Theorem}[section]
\newtheorem{corollary}[theorem]{Corollary}
\newtheorem{lemma}[theorem]{Lemma}
\newtheorem{proposition}[theorem]{Proposition}

\theoremstyle{definition}
\newtheorem{definition}[theorem]{Definition}
\newtheorem{example}[theorem]{Example}

\theoremstyle{remark}
\newtheorem{remark}[theorem]{Remark}

\newcommand{\thmcitemore}[2]{{\upshape \cite[#2]{#1}}}

\numberwithin{equation}{section}

\newcommand{\eref}[1]{(\ref{#1})}

\newcommand{\scr}{\mathcal}

 \DeclareMathOperator{\End}{End}

 \DeclareMathOperator{\sign}{sign}
 
 \DeclareMathOperator{\Ad}{Ad}
 \DeclareMathOperator{\ad}{ad}
 
 \DeclareMathOperator{\Ind}{Ind}
 \DeclareMathOperator{\Res}{Res}

 \DeclareMathOperator{\Ph}{Ph}

\DeclareMathOperator{\Spec}{Spec}

\newcommand{\st}{\; | \;}

\newcommand{\defeq}{:=}

\newcommand{\CC}{\mathbb{C}}

\newcommand{\NN}{\mathbb{N}}
\newcommand{\RR}{\mathbb{R}}

\newcommand{\ZZ}{\mathbb{Z}}

\newcommand{\half}{\frac{1}{2}}

\newcommand{\ip}[1]{\langle #1 \rangle}

\newcommand{\vspan}[1]{\mathrm{span}\{#1\}}

\newcommand{\slot}{\:\cdot\:}

\newcommand{\dual}{\dagger}

\newcommand{\pddiff}[3][.]{\ifthenelse{\equal{#1}{.}}
        {\frac{\partial^2 #2}{\partial #3^2}}
        {\frac{\partial^2 #1}{\partial #2 \partial #3}}}

\newcommand{\Lie}{\mathsf}
\newcommand{\lie}{\mathfrak}

\newcommand{\Univ}[2][.]{\ifthenelse{\equal{#1}{.}}{\mathcal{U}}{\mathcal{U}^{(#1)}}(\mathfrak{#2})}

 \DeclareMathOperator{\SL}{SL}
 \DeclareMathOperator{\SO}{SO}
 \DeclareMathOperator{\SU}{SU}
 \DeclareMathOperator{\Sp}{Sp}

\newcommand{\glx}{\mathfrak{gl}}
\newcommand{\slx}{\mathfrak{sl}}

\newcommand{\Weights}{\Lambda_W}
\newcommand{\Roots}{\Delta}

\newcommand{\coset}[1]{\underline{#1}}

\newcommand{\smatrix}[1]{\begin{pmatrix} #1 \end{pmatrix}}

\newcommand{\GTs}[6]{\left( \begin{array}{cccccc}
  \multicolumn{2}{c}{#1} &
  \multicolumn{2}{c}{#2} &
  \multicolumn{2}{c}{#3} \\
& \multicolumn{2}{c}{#4} &
  \multicolumn{2}{c}{#5} \\
&&\multicolumn{2}{c}{#6}
\end{array} \right)}

\newcommand{\pattern}{\Lambda}

\newcommand{\yvec}{\mathrm{y}}

\newcommand{\Phase}{\Ph}

\newcommand{\scrL}{\scr{L}}
\newcommand{\scrH}{\scr{H}}
\newcommand{\multop}[1]{M_{#1}}

\newcommand{\length}{l}
\newcommand{\Weylgroup}{\Lie{W}}

\newcommand{\edge}[3][]{{#2\xleftrightarrow{#1}#3}}
\newcommand{\shiftedaction}{\star}
\newcommand{\reflection}[1]{w_{#1}}

\newcommand{\scrX}[1][]{\scr{X}_{#1}}
\newcommand{\CX}{C(\scrX)} 
\newcommand{\LXE}[1][]{L^2(\scrX;E_{#1})}
\newcommand{\CXE}[1][]{C(\scrX;E_{#1})}
\newcommand{\CinftyXE}[1][]{C^\infty(\scrX;E_{#1})}

\newcommand{\LXEprime}[1][]{L^2(\scrX;E')}

\newcommand{\Lhol}[1]{L^\mathrm{hol}_{#1}}
\newcommand{\CXL}[1][]{C(\scrX;L_{#1})}
\newcommand{\LXL}[1][]{L^2(\scrX;L_{#1})}

\newcommand{\fibration}[1][]{\varphi_{#1}}
\newcommand{\foliation}[1][]{{\scr{F}_{#1}}}
\newcommand{\CP}{\CC \mathrm{P}}

\newcommand{\sect}{s}
\newcommand{\sectionpoi}{\varphi}

\newcommand{\cosphere}{S^*}
\newcommand{\Tangent}{T}
\newcommand{\Derivative}{{D}}
\newcommand{\Dolbeault}{\overline{\partial}}

\newcommand{\irrep}[1]{\hat{\Lie{#1}}}
\newcommand{\Khat}{\irrep{K}}

\newcommand{\triv}{{\sigma_0}}

\newcommand{\scrK}[1][]{\mathcal{K}_{#1}}
\newcommand{\scrJ}[1][]{\mathcal{J}_{#1}}
\newcommand{\scrA}[1][]{\mathcal{A}_{#1}}

\newcommand{\PsiDO[1]}[]{\Psi_{\foliation{#1}}}
\newcommand{\Symbol}[1][.]{\mathrm{Symb}_{ \ifthenelse{ \equal{#1}{.} }{}{{\foliation[#1]}} }}

\newcommand{\chiM}{\chi_{\Lie{M}}}
\newcommand{\chiA}{\chi_{\Lie{A}}}

\newcommand{\irrepset}{F}

\newcommand{\matrixunit}[2]{c_{#1,#2}}

\newcommand{\intertwiner}[3]{I_{#1,{#2}\to{#3}}}
\newcommand{\BGG}[3]{T_{#1,{#2}\to{#3}}}
\newcommand{\BGGextended}[3]{\tilde{T}_{#1,{#2}\to{#3}}}
\newcommand{\poIBGG}[3]{F_{#1,{#2}\to{#3}}}
\newcommand{\poI}[2]{N_{{#1}\to{#2}}}
\newcommand{\component}[1]{Q_{#1}}
\newcommand{\Fredholm}[1]{F_{#1}}

\newcommand{\Kcycle}[1]{\theta_{#1}}

\newcommand{\KK}{KK}
\newcommand{\K}{K}
\newcommand{\maptopt}{\iota}

\newcommand{\ulmatrix}[4]{
     \renewcommand{\baselinestretch}{1}\small\normalsize
     \left( \begin{array}{ccc}
     &&#2\\ \multicolumn{2}{c}{\raisebox{1.5ex}[0pt]{$#1$}} & #3 \\ #4
     \end{array} \right)
     \renewcommand{\baselinestretch}{2}\small\normalsize
 }
\newcommand{\drmatrix}[4]{
    \renewcommand{\baselinestretch}{1}\small\normalsize
    \left( \begin{array}{ccc}
    #1\\
    #2\\
    #2 & \multicolumn{2}{c}{\raisebox{1.5ex}[0pt]{$#4$}}
    \end{array} \right)
    \renewcommand{\baselinestretch}{2}\small\normalsize\vspace{2ex}
 }

\begin{document}

\title{The Bernstein-Gelfand-Gelfand complex and Kasparov theory for $\SL(3,\CC)$}
\author{Robert Yuncken}

\maketitle
\begin{abstract}
For $\Lie{G}=\mathrm{SL}(3,\mathbb{C})$, we construct an element of $\Lie{G}$-equivariant analytic $K$-homology from the Bernstein-Gelfand-Gelfand complex for $\Lie{G}$.  This furnishes an explicit splitting of the restriction map from the Kasparov representation ring $R(\Lie{G})$ to the representation ring $R(\Lie{K})$ of its maximal compact subgroup $\SU(3)$, and the splitting factors through the equivariant $K$-homology of the flag variety $\scrX$ of $\Lie{G}$.  In particular, we obtain a new model for the $\gamma$-element of $\Lie{G}$.

The construction is made using $\SU(3)$-harmonic analysis associated to the canonical fibrations of $\scrX$.  On this matter, we prove results which demonstrate the compatibility of both the $\Lie{G}$-action and the order zero longitudinal pseudodifferential operators with the $\SU(3)$-harmonic analysis. 

\end{abstract}


\section{Introduction}

A typical source for the construction of a Kasparov $K$-homology
cycle is an elliptic differential complex.  If the elliptic complex
is equivariant with respect to the action of a group $\Lie{G}$, and
if moreover the group action satisfies an additional conformality
property (see below), then one can obtain an element of equivariant $K$-homology. But if
$\Lie{G}$ is a semisimple Lie group of rank greater than one, non-trivial examples of such complexes cannot
exist (\cite{Puschnigg}). This paper describes a means of constructing an equivariant
$K$-homology class from the
Bernstein-Gelfand-Gelfand complex for
$\SL(3,\CC)$---a differential complex which is neither elliptic nor
conformal, but which satisfies some weaker (`directional') form of
these conditions.

The motivation for this construction comes from the Baum-Connes
conjecture. Although an understanding of the conjecture is not
essential to this paper, it is useful for perspective.  The conjecture asserts that for a second countable locally compact group $\Lie{G}$, the 
 {\em assembly map}
$$
  \mu_{\Gamma} : K^\Gamma(\underline{E}\Gamma) \to K(C^*_r\Gamma).
$$
is an isomorphism, thus giving a `topological computation' of the $K$-theory of the reduced group $C^*$-algebra.  For a fuller description of the conjecture and its many consequences, we refer the reader to the expository article \cite{HigsonICM} and the foundational paper \cite{BCH}.

The conjecture has been proven for a wide class of groups, amongst which we mention in particular the discrete subgroups of simple Lie groups of real rank one. A notable unknown, however, is the group $\SL(3,\ZZ)$.  More broadly, the conjecture is unknown for general discrete subgroups of semisimple Lie groups of rank greater than one.

For subgroups of rank one semisimple groups $\Lie{G}$, the proofs in each case centre on a canonical idempotent $\gamma$ in the representation ring $R(\Lie{G})\defeq\KK^\Lie{G}(\CC,\CC)$. (See \cite{Kas-Lorentz} for $\Lie{G}=\SO_0(n,1)$, \cite{JK} for $\Lie{G}=\SU(n,1)$, \cite{Julg}\footnote{The first proof of Baum-Connes for discrete subgroups of $\Sp(n,1)$ was due to V.~Lafforgue, but used a somewhat different approach.} for $\Sp(n,1)$).  For our purposes, the most convenient way to describe this idempotent $\gamma$ is via the following fact.

\begin{theorem}[Kasparov]
\label{thm:split_surjection}
Let $\Lie{G}$ be a semisimple Lie group and $\Lie{K}$ a maximal compact subgroup.  The restriction map $R(\Lie{G}) \to R(\Lie{K})$ is a split surjection of rings.
\end{theorem}

The unit in $R(\Lie{K})$ is the class of the trivial $\Lie{K}$-representation, and its image under the splitting is an idempotent in $R(\Lie{G})$.  This is $\gamma$.  

If $\gamma=1\in R(\Lie{G})$ then the restriction map is an isomorphism.  In this case, the `Dirac-dual Dirac method' of Kasparov implies that the Baum-Connes conjecture holds for all discrete subgroups of $\Lie{G}$.  This is the approach taken in the papers cited above, although in the case of $\Sp(n,1)$ a weaker notion of `triviality' for $\gamma$ must be used.

The idempotent $\gamma$ was originally defined via equivariant $\K$-homology for the proper $\Lie{G}$-space $\Lie{G/K}$ (\cite{Kas88}).  In the rank-one proofs mentioned above, however, $\gamma$ is more conveniently constructed using the compact space $\Lie{G/B}$, where $\Lie{B}$ is a minimal parabolic subgroup.  This can be explained by the fact that the induced representations from $\Lie{B}$ give a natural topological parameterization of (the relevant subset of) representations of $\Lie{G}$, namely the generalized principal series, including the complementary series.  It is also pertinent that $\Lie{B}$ is amenable, so itself satisfies Baum-Connes.

It is instructive to consider the construction of $\gamma$ in the simple example $\Lie{G}=\SL(2,\CC)$.  One begins with the Dolbeault complex for the homogeneous space $\Lie{G}/\Lie{B} \cong \CP^1$:
$$
  \Omega^{0,0}\CP^1 \xrightarrow{\Dolbeault} \Omega^{0,1}CP^1.
$$   
This is a $\Lie{G}$-equivariant elliptic complex.  Importantly, though, $\CP^1$ does not admit a $\Lie{G}$-invariant Riemannian metric.  The action is conformal (with respect to the natural $\Lie{K}$-equivariant metric), and the translation representation of $\Lie{G}$ on $L^2\Omega^{0,\bullet}\CP^1$ can be made unitary by the introduction of a scalar Radon-Nikodym factor.  But the operator $D \defeq \Dolbeault + \Dolbeault^*$ will not be $\Lie{G}$-equivariant, not even in the weak sense of defining an unbounded equivariant Fredholm module.  Somewhat magically though, replacing $D$ by its operator phase results in a bounded equivariant Fredholm module.  For this to work it is crucial that the $\Lie{G}$-action is conformal\footnote{In general, the conformality requirement is even stronger:  the ratio of the Radon-Nikodym factors in degrees $p$ and $p+1$ must be independent of $p$.  We will not explain this further.} on the Hermitian bundles $\Omega^{0,p}\CP^1$.

In order to maintain this crucial conformality property for the other rank one cases, one must use increasingly complicated subellitpic differential complexes --- the Rumin complex for $\SU(n,1)$; a quaternionic analogue thereof for $\Sp(n,1)$ --- and corresponding nonstandard pseudodifferential calculi.  We remark that $\K$-homological constructions using even nonstandard pseudodifferential calculi typically result in finitely-summable Fredholm modules.  Puschnigg \cite{Puschnigg} has shown that simple Lie groups of higher rank do not admit any nontrivial finitely summable Fredholm modules.

This motivates our construction using the Bernstein-Gelfand-Gelfand (`BGG') complex.

\begin{theorem}[Bernstein-Gelfand-Gelfand]
Let $\Lie{G}$ be a complex semisimple group and $\Lie{B}$ a minimal parabolic subgroup.  For any finite dimensional holomorphic representation $V$ of $\Lie{G}$ there is a differential complex, consisting of direct sums of homogeneous line bundles over $\Lie{G/B}$ and $\Lie{G}$-equivariant differential operators between them, which resolves $V$.
\end{theorem}

The bundles in each degree here are not conformal, but their component line bundles are individually conformal.   (Trivially, any group action on a Hermitian line bundle is conformal.)  The question is whether this structure is enough to produce an element of equivariant $\K$-homology.  In this paper, we answer this question affirmatively in the case of $\Lie{G}=\SL(3,\CC)$.  We thereby obtain an explicit construction of the splitting map $ R(\Lie{K}) \to R(\Lie{G})$, and in particular a construction of $\gamma$, which factors through $\KK^\Lie{G}( C(\Lie{G/B}), \CC)$.

The construction is based upon harmonic analysis of $\SU(3)$ rather than some nonstandard pseudodifferential calculus.  An indication of the difficulties of a purely pseudodifferential approach is given in Chapter 5 of \cite{Yuncken-thesis}.  In fact, our construction could be made without any reference to pseudodifferential operators at all, though pseudodifferential theory has become so central to index theory that to do so might seem somewhat eccentric.

Much of the required harmonic analysis has been developed in \cite{Yuncken:PsiDOs} in the broader context of $\SU(n)$ ($n\geq2$).  We expect that the results of this paper should be extendable the groups $\SL(n,\CC)$, and indeed to complex semisimple groups in general.  The main technical difficulty in the case of $\SL(n,\CC)$ is an appropriate version of the the operator partition of unity of Lemma \ref{lem:operator_po1} of this paper.  For general semisimple groups, the required directional harmonic analysis is yet to be developed.


As for the Baum-Connes Conjecture itself, it is known that $\gamma\neq1$ for any group $\Lie{G}$ which has Kazhdan's property $T$. Therefore,
a direct translation of Kasparov's method cannot prove the
Baum-Connes conjecture for simple Lie groups of rank greater than
one---some subtle variation of Kasparov's argument would be
required. Nevertheless, it is expected that the present construction
will be useful for further study of the Baum-Connes conjecture.

\medskip

Let us now describe the BGG complex in more detail.  In fact, knowledge of the cohomological version of the BGG complex is unnecessary for the present paper, since our $\K$-homological version will be produced from scratch.  
But it is such a strong motivation that it is worth spending some time explaining it.

Finite dimensional holomorphic representations of $\Lie{G}$ are parameterized by their highest weights.  Let $V^\lambda$ denote the representation with highest weight $\lambda$.  Any weight $\mu$ of $\Lie{G}$ extends to a holomorphic character of $\Lie{B}$ (see Section \ref{sec:homogeneous_vector_bundles}), and we denote by $\Lhol{\mu}$ the corresponding induced holomorphic line bundle over $\scrX\defeq\Lie{G/B}$.  The Borel-Weil Theorem states that $V^\lambda$ is equivariantly isomorphic to the space of global holomorphic sections of $\Lhol{\lambda}$.  

Recall that the Weyl group $\Weylgroup$ is a group of reflections on the weight space.  It is generated by the {\em simple reflections}---reflections in the walls orthogonal to a choice of simple roots for $\Lie{G}$.  Word length in these generators defines a length function $\length:\Weylgroup\to\NN$.  We need the {\em shifted action} of the Weyl group defined by the formula $w\shiftedaction \mu \defeq w(\mu+\rho) - \rho$, where $\rho$ is the half-sum of the positive roots.  Bernstein, Gelfand and Gelfand \cite{BGG} showed that there is a holomorphic $\Lie{G}$-equivariant differential operator from $\Lhol{\mu}$ to $\Lhol{\nu}$ if and only if $\mu=w\shiftedaction\lambda$ and $\nu = w'\shiftedaction\lambda$ for some dominant weight $\lambda$ and some $w,w'\in\Weylgroup$ with $\length(w')\geq\length(w)$.  What is more, these operators can be assembled into an exact complex as follows\footnote{Strictly speaking, Bernstein, Gelfand and Gelfand made a homological complex by assembling intertwiners between Verma modules.  What we are calling the BGG complex here is a dual cohomological complex.  See the appendix of \cite{CSS} for an explanation of this. }.  One defines the degree $p$ cocycle space $C^p \defeq \bigoplus_{\length(w) = p} C^\infty(\scrX, \Lhol{w\shiftedaction \lambda})$.   The collection of equivariant differential operators between any $\Lhol{w\shiftedaction\lambda}$ and $\Lhol{w'\shiftedaction\lambda}$ with $\length(w)=p$, $\length(w')=p+1$ defines a matrix of operators $C^p \to C^{p+1}$.  With an appropriate choice of signs these operators resolve the Borel-Weil inclusion $V^\lambda \hookrightarrow C^\infty(\scrX;\Lhol{\lambda})$.

In the case of $\SL(3,\CC)$, we get a complex
\begin{equation}
\label{eq:BGG_resolution}
 \xymatrix@!C=7.2ex{
 & C^\infty(\scrX;\Lhol{\reflection{\alpha_1}\shiftedaction\lambda}) \ar[rr]\ar[ddrr]
   \ar@{.}[dd]|-{\bigoplus}
  && C^\infty(\scrX;\Lhol{\reflection{\alpha_1}\reflection{\alpha_2}\shiftedaction\lambda}) \ar[dr]
   \ar@{.}[dd]|-{\bigoplus}
 \\
 V^\lambda \hookrightarrow C^\infty(\scrX;\Lhol{\lambda})\ar[ur]\ar[dr]\quad\quad
 &&&&\quad\quad C^\infty(\scrX;\Lhol{w_\rho\shiftedaction\lambda}) \\
 & C^\infty(\scrX;\Lhol{\reflection{\alpha_2}\shiftedaction\lambda})\ar[rr]\ar[uurr]
  && C^\infty(\scrX;\Lhol{\reflection{\alpha_2}\reflection{\alpha_1}\shiftedaction\lambda})\ar[ur] \
  }
\end{equation}
where $\alpha_1$, $\alpha_2$ and $\rho = \alpha_1+\alpha_2$ are the positive roots, and $\reflection{\alpha}$ denotes the reflection in the wall orthogonal to $\alpha$.

In this paper, we define a `normalized', {\em i.e.}, $L^2$-bounded, version of this complex which is analogous to the equivariant Fredholm module constructed above from the Dolbeault complex of $\CP^1$.

\medskip

To complete this overview, we give a very brief description of the harmonic analysis upon which our $\K$-homological BGG construction is based.  The space $\scrX\defeq \Lie{G/B}$ is the complete flag variety of $\CC^3$.   Corresponding to the simple roots $\alpha_1$ and $\alpha_2$, there are $\Lie{G}$-equivariant fibrations $\scrX\to\scrX[i]$ $(i=1,2)$ where $\scrX[1]$ and $\scrX[2]$ are the Grassmannians of lines and planes in $\CC^3$.  As described in \cite{Yuncken:PsiDOs}, associated to each of these fibrations is a $C^*$-algebra $\scrK[\alpha_i]$ of operators on the $L^2$-section space of any homogeneous line bundle over $\scrX$.  This algebra contains, in particular, the longitudinal pseudodifferential operators of negative order tangent to the given fibration.  A key property is that the intersection $\scrK[\alpha_1]\cap \scrK[\alpha_2]$ consists of compact operators.  Ultimately, this allows us to apply the Kasparov Technical Theorem to construct a Fredholm module from the normalized BGG operators.

\medskip

The structure of the paper is as follows.  Section \ref{sec:preliminaries} gives the background on the structure theory of the semisimple Lie group $\Lie{G}=\SL(3,\CC)$, the flag variety $\scrX$ and its homogeneous line bundles, mainly for the purpose of setting notation.  

In Section \ref{sec:harmonic_analysis} we review the $C^*$-algebras $\scrK[\alpha_i]$ of \cite{Yuncken:PsiDOs} and their relation to longitudinal pseudodifferential operators on the flag variety $\scrX$.  We also prove two important new results concerning these algebras.  For the sake of stating these results elegantly, it is convenient to place the $C^*$-algebras $\scrK[\alpha_i]$ in the context of $C^*$-categories (see Section \ref{sec:categories} for details).

\begin{theorem}
\label{thm:harmonic_analysis_results}
Let $E$, $E'$ be $\Lie{G}$-homogeneous line bundles over $\scrX$.  Let $\scrA$ denote the simultaneous multiplier category of $\scrK[\alpha_1]$ and $\scrK[\alpha_2]$ (see Definition \ref{def:A}).
\begin{enumerate}
\item  The translation operators $g: \LXE \to \LXE$ belong to $\scrA$, for all $g\in\Lie{G}$.
\item  If $T:\LXE \to L^2(\scrX;E')$ is a longitudinal pseudodifferential operator of order zero tangent to one of the fibrations $\scrX\to\scrX[i]$ ($i=1,2$), then $T\in\scrA$.
\end{enumerate}
\end{theorem}

Theorem \ref{thm:harmonic_analysis_results}(i) is proven in Section \ref{sec:principal_series}.  Part (ii) is restated in Theorem \ref{thm:PsiDOs_in_A}.  The proof requires some lengthy computations in $\SU(3)$ harmonic analysis which are presented in Appendix \ref{sec:PsiDOs_in_A}.  

In Section \ref{sec:construction}, we combine the above results to construct an element of  $\KK^\Lie{G}(C(\scrX),\CC)$ from the BGG complex.  We also explain why this yields the splitting of the restriction morphism $R(\Lie{G}) \to R(\Lie{K})$.

\medskip

Part of this work appeared in the author's doctoral dissertation
\cite{Yuncken-thesis}. I would like to thank my thesis adviser, Nigel Higson.  I would also like to thank Erik Koelink for
several informative conversations.

\section{Notation and Preliminaries}
\label{sec:preliminaries}

\subsection{Lie groups}
\label{sec:Lie_groups}

Throughout this paper $\Lie{G}$ will denote the group $\SL(3,\CC)$.  We fix notation for the following subgroups:   $\Lie{K}=\SU(3)$, its maximal compact subgroup; $\Lie{H}$, the Cartan subgroup of diagonal matrices; $\Lie{A}$, the subgroup of diagonal matrices with positive real entries; $\Lie{M}  = \Lie{H}\cap\Lie{K}$, the maximal torus of $\Lie{K}$; $\Lie{N}$, the subgroup of upper triangular unipotent matrices; and $\Lie{B}=\Lie{MAN}$ the subgroup of upper triangular matrices.  Their Lie algebras are denoted  $\lie{g}$, $\lie{k}$, $\lie{h}$, $\lie{a}$, $\lie{m}$, $\lie{n}$ and $\lie{b}$.   

We use $V^\dual$ to denote the dual of a complex vector space $V$.   
We make the usual identifications of the complexifications $\lie{m}_\CC$ and $\lie{a}_\CC$ with $\lie{h}$ by extending the inclusions $\lie{a},\lie{m}\hookrightarrow\lie{h}$ to $\CC$-linear maps.  We thereby identify characters of $\Lie{A}$ and $\Lie{M}$ with elements of $\lie{h}^\dual$.  Characters of $\lie{h}$ will be denoted by $\chi = \chiM \oplus \chiA$, where $\chiM$ and $\chiA$ are the restrictions of $\chi$ to $\lie{m}$ and $\lie{a}$, respectively.  The corresponding group character of $\Lie{H}$ will be denoted $e^\chi$.  The weight lattice in $\lie{m}_\CC^\dual \cong \lie{h}^\dual$ will be denoted by $\Lambda_W$.

The set of roots of $\Lie{K}$ is denoted $\Delta$.  We fix the notation
$$
X_{\alpha_1} = \left(\begin{array}{ccc} 0&1&0\\0&0&0\\0&0&0 \end{array}\right),~
X_{\alpha_2} = \left(\begin{array}{ccc} 0&0&0\\0&0&1\\0&0&0 \end{array}\right),~
X_\rho = \left(\begin{array}{ccc} 0&0&1\\0&0&0\\0&0&0 \end{array}\right)~
\in \lie{k}_\CC \cong \lie{g},
$$
which are root vectors for the roots $\alpha_1$, $\alpha_2$ and $\rho\defeq\alpha_1+\alpha_2$.  We fix these as our set of positive roots $\Delta^+$, so $\Sigma \defeq \{\alpha_1, \alpha_2\}$ is the set of simple roots.  For each $\alpha\in\Delta^+$, $Y_\alpha$ will denote the transpose of $X_\alpha$.  We abbreviate $X_{\alpha_i}$ and $Y_{\alpha_i}$ to $X_i$ and $Y_i$, whenever convenient.

We put $H_i \defeq [X_i, Y_i] \in \lie{m}_\CC$.  The elements $X_i, Y_i, H_i$ span a Lie subalgebra isomorphic to $\slx(2,\CC)$, which we denote by $\lie{s}_i$.  We also put
$$
  H_1' \defeq \left(\begin{array}{ccc} 1&0&0\\0&1&0\\0&0&-2 \end{array}\right), \qquad
  H_2' \defeq \left(\begin{array}{ccc} -2&0&0\\0&1&0\\0&0&1 \end{array}\right) \qquad 
  \in \lie{m}_\CC \cong \lie{h},
$$
so that for fixed $i=1,2$, $H_i$ and $H_i'$ span $\lie{h}$ and $H_i'$ commutes with $\lie{s}_i$.

The Weyl group of $\Lie{G}$ is $\Lie{W}\cong S_3$.  We let $\reflection{\alpha}$ denote the reflection in the wall orthogonal to the root $\alpha$.  The {simple reflections} $w_{\alpha_1}$ and $w_{\alpha_2}$ are generators of $W$, and the minimal word length in these generators defines the length function $\length$ on $W$.


\subsection{Homogeneous vector bundles}
\label{sec:homogeneous_vector_bundles}

Throughout, $\scrX$ will denote the homogeneous space $\scrX=\Lie{G}/\Lie{B} = \Lie{K}/\Lie{M}$.

Let $\chi = \chiM\oplus\chiA$ be a character of $\lie{h}$.  As usual, we extend it trivially on $\lie{n}$ to a character of $\lie{b}$.  We use $L_\chi$ to denote the $\Lie{G}$-homogeneous line bundle over $\scrX$ which is induced from $\chi$.  That is, continuous sections of $L_\chi$ are identified with $\Lie{B}$-equivariant functions on $\Lie{G}$ as follows:
\begin{multline}
  \CXL[\chi] = \{ \sect:\Lie{G}\to\CC \text{ continuous }\st \sect (gman) = e^{\chiM}(m^{-1})e^{\chiA}(a^{-1}) \sect (g) 
     \\
   \forall g\in\Lie{G}, m\in\Lie{M}, a\in\Lie{A}, n\in\Lie{N} \}.
\label{eq:B-equivariance}
\end{multline}
The $\Lie{G}$-action on sections is by left translation:  $g'\cdot\sect(g) \defeq \sect({g'}^{-1}g)$.
Restricting to $\Lie{K}$, we have the `compact picture' of $\CXL[\chi]$:
\begin{multline}
  \CXL[\chi] \cong \{ \sect:\Lie{K} \to \CC \text{ continuous } \st \sect(km) = e^{\chiM}(m^{-1})s(k) 
     \\
   \forall k\in\Lie{K}, m\in\Lie{M} \}.
\label{eq:M-equivariance}
\end{multline}
Note that, as a $\Lie{K}$-homogeneous bundle, $L_\chi$ depends only on $\chiM$.

The compact picture gives a Hermitian metric on $L_\chi$.  Specifically, the pointwise inner product of sections is given by
$$
  \ip{\sect_1(k),\sect_2(k)} = \overline{\sect_1(k)} \sect_2(k) \quad \in C(\scrX).
$$
The $L^2$-section space $\LXL[\chi]$ is the completion of $\CXL[\chi]$ with respect to the inner product
\begin{equation}
\label{eq:inner_product}
  \ip{\sect_1,\sect_2} = \int_\Lie{K} \overline{\sect_1(k)} \sect_2(k) \, dk.
\end{equation}

Some cases warrant special notation.  If $\mu$ is a weight for $\Lie{K}$, we let $\Lhol{\mu}$ denote the holomorphic line bundle $L_{\mu\oplus\mu}$.  We also let $E_\mu$ denote the `unitarily induced' bundle $L_{\mu\oplus\rho}$. On $E_\mu$ the translation action $U_\mu:\Lie{G} \to \scrL(\LXE[\mu])$ is a {unitary} representation.  These will be the main focus of our attention.

Restricting $U_\mu$ to $\Lie{K}$, $\LXE[\mu]$ becomes a subrepresentation of the left regular representation $\Lie{K}$.  If $R$ denotes the \emph{right} regular representation, then the equivariance condition of Equation \eref{eq:M-equivariance} becomes $R(m) \sect = e^{-\mu}(m) \sect$ for all $m\in\Lie{M}$.  Infinitesimally,
\begin{eqnarray}
  \LXE[\mu] &=& \{\sect\in L^2(\Lie{K}) \st R(M) \sect = -\mu(M) \sect \text{ for all } M\in\lie{m} \} \nonumber \\
    &=& p_{-\mu} L^2(\Lie{K}),
\label{eq:m-equivariance}
\end{eqnarray}
where $p_{-\mu}$ denotes the orthogonal projection onto the $(-\mu)$-weight space of the {\em right} regular representation of $\Lie{K}$ on $L^2(\Lie{K})$.

Let $\chi$, $\chi'$ be characters of $\Lie{B}$.  If $f\in\CXL[\chi'-\chi]$ then pointwise multiplication by $f$, denoted $\multop{f}$, maps $\CXL[\chi]$ to $\CXL[\chi']$.  This gives a $\Lie{G}$-equivariant bundle isomorphism $\End(L_\chi,L_{\chi'}) \cong L_{\chi'-\chi}$.  In particular, $\End(E_\mu, E_{\mu'}) \cong L_{(\mu'-\mu)\oplus0}$ for any weights $\mu$, $\mu'$.  Moreover, for any $f\in\CXL[(\mu'-\mu)\oplus0]$,
\begin{equation}
\label{eq:covariance_of_multops}
  U_{\mu'}(g) \multop{f} U_\mu(g^{-1}) = \multop{g\cdot f}.
\end{equation}
In this picture, a locally trivializing partition of unity on $E_\mu$ takes the following form.

\begin{lemma}
\label{lem:partition_of_unity}

For any weight $\mu$, there exists a finite collection of continuous sections $\sectionpoi_1,\ldots,\sectionpoi_n \in \CXL[\mu\oplus0]$ such that $\sum_{j=1}^n \multop{\sectionpoi_j} \multop{\overline{\sectionpoi_j}} = 1$.

\end{lemma}

\begin{proof}
Let $f_1,\ldots,f_n\in \CX$ be a partition of unity subordinate to a locally trivializing cover of $E_\mu$.   Composing $f_j^{\half}$ with the corresponding local trivialization $L_0 \xrightarrow{\cong} L_{\mu\oplus0}$ gives the sections $\sectionpoi_j$.
\end{proof}


\subsection{Parabolic subgroups and equivariant fibrations}
\label{sec:parabolic_subgroups}

Let $\Lie{P}$ be a parabolic subgroup, $\Lie{B}\leq \Lie{P} \leq \Lie{G}$, with Lie algebra $\lie{p}$.  Let $S\subseteq \Sigma$ be the set of simple roots $\alpha$ such that the root space $\lie{g}_{-\alpha}$ is contained in $\lie{p}$.  This set classifies $\Lie{P}$, and we therefore introduce the notation 
$$
  \Lie{P}_{\Sigma} \defeq \Lie{G},  \quad
  \Lie{P}_{\{\alpha_1\}} \defeq \left\{ \smatrix{ *&*&*\\ *&*&*\\0&0&* } \right\}, \quad
  \Lie{P}_{\{\alpha_2\}} \defeq \left\{ \smatrix{ *&*&*\\ 0&*&*\\0&*&* } \right\}, \quad
  \Lie{P}_{\emptyset} \defeq \Lie{B}.
$$
(Here $*$ denotes possibly nonzero entries.)  We will simplify this by writing $\Lie{P}_i \defeq \Lie{P}_{\{\alpha_i\}}$ whenever convenient.

For $i=1,2$, let $\scrX[i] \defeq \Lie{G}/\Lie{P}_i$.  The natural maps $\fibration[i] : \scrX \to \scrX[i]$ are equivariant fibrations with fibres $\Lie{P}_i/\Lie{B} \cong \CP^1$.  We will denote the corresponding foliations of $\scrX$ by $\foliation[i]\defeq \ker D\fibration[i]$.

Denote the compact part of $\Lie{P}_S$ by $\Lie{K}_S \defeq \Lie{P}_S \cap \Lie{K}$. Explicitly,
\begin{eqnarray*}
  \Lie{K}_\Sigma &\defeq& \Lie{K}, \\
  \Lie{K}_1 &\defeq&  \Lie{P}_1 \cap \Lie{K}  
    = \left\{ \ulmatrix{A}{0}{0}{0&0&z} \bigg| \quad
     A\in\mathrm{U}(2), ~  z = (\det A)^{-1}  \right\}, \\
  \Lie{K}_2 &\defeq&  \Lie{P}_2 \cap \Lie{K}  
    = \left\{ \drmatrix{z&0&0}{0}{0}{A} \bigg| \quad
     A\in\mathrm{U}(2), ~  z = (\det A)^{-1}  \right\}, \\
  \Lie{K}_\emptyset &\defeq& \Lie{M}.
\end{eqnarray*}
Then $\scrX[i]=\Lie{K}/\Lie{K}_i$ ($i=1,2$).

The complexified Lie algebra $(\lie{k}_i)_\CC$ of $\Lie{K}_i$ decomposes as $\lie{s}_i\oplus\lie{z}_i$, where $\lie{s}_i \defeq \vspan{X_i, H_i, Y_i} \cong \slx(2,\CC)$ and $\lie{z}_i \defeq \vspan{H_i'}\subset \lie{m}_\CC$.  (Notation as in Section \ref{sec:Lie_groups}.)    For the sake of fixing notation, we recall the representation theory of $\lie{s}_i \cong \slx(2,\CC)$.  The weights of $\slx(2,\CC)$ are parameterized by the integers.  The restriction of a weight $\mu$ of $\Lie{K}$ to a weight of $\lie{s}_i$ is $\mu_i \defeq \mu(H_i) \in \ZZ$.  The dominant weights are the nonnegative integers $\NN$.  

Let $X,H,Y\in\slx(2,\CC)$ be the basis elements corresponding to $X_i, H_i. Y_i\in\lie{s}_i$.  The irreducible representation of $\slx(2,\CC)$ with highest weight $\delta\in\NN$ will be denoted $V^\delta$.  It has an orthonormal basis of weight vectors $\{ e_\delta, e_{\delta-2}, \ldots, e_{-\delta+2}, e_{-\delta} \}$, such that
\begin{eqnarray}
\label{eq:X-formula}
  X\cdot e_j &=& \half \sqrt{(\delta-j)(\delta+j+2)} \,e_{j+2} \\
\label{eq:H-formula}
  H\cdot e_j &=&  j\, e_{j} \\
\label{eq:Y-formula}
  Y\cdot e_j &=& \half \sqrt{(\delta-j+2)(\delta+j)} \, e_{j-2} 
\end{eqnarray}

\subsection{Harmonic analysis}
\label{sec:harmonic_notation}

For any compact group $\Lie{C}$, we will use $\irrep{C}$ to denote the set of irreducible representations of $\Lie{C}$, often referred as {\em $\Lie{C}$-types}.  For any unitary representation $\pi$ of $\Lie{C}$, we use $V^\pi$ to denote its representation space, and $\pi^\dual$ to denote its contragredient representation.

For a representation $\pi$ of $\Lie{K}=\SU(3)$ and elements $\xi\in V^\pi$, $\eta^\dual\in V^{\pi\dual}$, we use $\matrixunit{\eta^\dual}{\xi}$ to denote the matrix unit $\matrixunit{\eta^\dual}{\xi}(k) \defeq (\eta^\dual, \pi(k)\xi)$.  
Recall the Peter-Weyl isomorphism
\begin{eqnarray*}
  \bigoplus_{\pi\in\Khat} V^{\pi\dual} \otimes V^\pi & \cong & L^2(\Lie{K}) \\
    \eta^\dual \otimes \xi & \mapsto & (\dim V^\pi)^\half \matrixunit{\eta^\dual}{\xi}.
\end{eqnarray*}
which intertwines $\bigoplus \pi$ and $\bigoplus\pi^\dual$ with the left and right regular representations, respectively.
If $p_\mu$ denotes the projection onto the $\mu$-weight space of a representation then from Equation \ref{eq:m-equivariance},
$$
  \LXE[\mu] \cong  \bigoplus_{\pi\in\Khat} V^{\pi\dual} \otimes p_{-\mu}V^\pi  .
$$


\section{Harmonic analysis on the flag variety}
\label{sec:harmonic_analysis}

\subsection{Harmonic $C^*$-categories}
\label{sec:harmonic_decompositions}
\label{sec:categories}

We will make much use of the results of \cite{Yuncken:PsiDOs} regarding harmonic analysis on flag manifolds for $\SL(n,\CC)$.  In this section, we review the major definitions and results of that paper.  Because we are only interested in $n=3$ here, we will simplify the notation somewhat.  

Let $\Lie{K}'$ be a closed subgroup of $\Lie{K}=\SU(3)$.  Let $\scrH$ be a Hilbert space equipped with a unitary representation of $\Lie{K}$.  For $\sigma\in\Khat'$, we let $p_\sigma$ denote the orthogonal projection onto the $\sigma$-isotypical subspace of $\scr{H}$ (with representation restricted to $\Lie{K}'$).  
If $\irrepset\subseteq \Khat'$ is a set of $\Lie{K}'$-types, we let $p_F \defeq \sum_{\sigma\in\irrepset} p_\sigma$. 

We are particularly interested in the four subgroups $\Lie{K} \geq \Lie{K}_1, \Lie{K}_2, \geq \Lie{M}$ above.   Note that the isotypical subspaces of $\Lie{M}$ are the weight spaces.

If $\Lie{K}''$ is a subgroup of $\Lie{K}'$, then the isotypical projections of $\Lie{K}'$ and $\Lie{K}''$ commute.  In particular, the isotypical projections of $\Lie{K}$, $\Lie{K}_1$ and $\Lie{K}_2$ commute with the weight-space projections.  These isotypical projections can therefore be restricted to any weight-space of a unitary $\Lie{K}$-representation.

\begin{definition}
A {\em harmonic $\Lie{K}$-space} $H$ is a direct sum of weight spaces of unitary $\Lie{K}$-representations:  $H = \bigoplus_{k} p_{\mu_k} \scr{H}_k$ for some weights $\mu_k$ and unitary $\Lie{K}$-representations on $\scr{H}_k$.

A harmonic $\Lie{K}$-space $H$ is called {\em finite multiplicity} if for every $\pi\in\Khat$, $p_\pi H$ is finite dimensional.
\end{definition}

\begin{example}
\label{ex:finite_multiplicities}
The (right) regular representation is a finite multiplicity harmonic $\Lie{K}$-space by the Peter-Weyl Theorem, as is $\LXE[\mu]$ for any weight $\mu$.  More generally, any homogeneous vector bundle $E$ over $\scrX$ decomposes equivariantly into line bundles, so $\LXE$ is a harmonic $\Lie{K}$-space.  
\end{example}

\begin{definition}
\label{def:A_K}
Let $S\subseteq\Sigma$.  Let $A:H \to H'$ be a bounded linear operator between harmonic $\Lie{K}$-spaces.  For $\sigma',\sigma\in\Khat_S$, let  $A_{\sigma'\sigma} \defeq p_\sigma' A p_\sigma$, so that $(A_{\sigma'\sigma})$ is the matrix decomposition of $A$ with respect to the decompositions of $H,H'$ into $\Lie{K}_S$-types.
\begin{enumerate}
\item We say $A$ is {\em $\Lie{K}_S$-harmonically proper} if the matrix $(A_{\sigma'\sigma})$ is row- and column-finite, {\em i.e.}, if for every $\sigma \in \Khat_S$, there are only finitely many $\sigma'\in \Khat_S$ for which either $A_{\sigma'\sigma}$ or $A_{\sigma\sigma'}$ is nonzero.
\item We say $A$ is {\em $\Lie{K}_S$-harmonically finite} if the matrix $(A_{\sigma'\sigma})$ has only finitely many nonzero entries.
\end{enumerate}

Define $\scrA[S](H,H')$, resp.~$\scrK[S](H,H')$, to be the operator-norm closure of the $\Lie{K}_S$-harmonically proper, resp.~$\Lie{K}_S$-harmonically finite, operators from $H$ to $H'$.  
\end{definition}

If $H=H'$, we write $\scrA[S](H)$  and $\scrK[S](H)$ for $\scrA[S](H,H)$  and $\scrK[S](H,H)$, respectively.  These are $C^*$-subalgebras of the algebras $\scrL(H)$ of bounded operators on $H$.  Letting $H$ and $H'$ vary, we consider $\scrA[S]$ and $\scrK[S]$ as defining $C^*$-categories of operators between harmonic $\Lie{K}$-spaces.
We also use $\scrK$ and $\scrL$ to denote the $C^*$-categories of compact operators and bounded operators, respectively, between Hilbert spaces.

\begin{lemma}[\thmcitemore{Yuncken:PsiDOs}{Lemma 3.2}]

If $S\subseteq S' \subseteq \Sigma$ then $\scrK[S'] \subseteq \scrK[S]$.

\end{lemma}

The following two results are restatements of Lemmas 3.4 and 3.5 of \cite{Yuncken:PsiDOs}.

\begin{proposition}
\label{prop:KS_equivalent_conditions}
Let $K:H \to H'$ be a bounded linear operator between harmonic $\Lie{K}$-spaces.  The following are equivalent:
\begin{enumerate}
\item  $K \in \scrK[S]$,
\item  For any $\epsilon>0$, there is a finite set $\irrepset\subset\Khat_S$ of $\Lie{K}_S$-types such that $\|p_\irrepset^\perp K\|<\epsilon$ and $\|K p_\irrepset^\perp \| < \epsilon$.
\item  For any $\epsilon>0$, there is a finite set $\irrepset\subset\Khat_S$ of $\Lie{K}_S$-types such that $\|K - p_\irrepset K p_\irrepset \|<\epsilon$.
\end{enumerate}
\end{proposition}

If $A$ and $K$ are bounded linear operators, we say $K$ is {\em right-composable} for $A$ if the codomain of $K$ is the domain of $A$.  {\em Left-composability} is defined similarly.

\begin{proposition}
\label{prop:AS_equivalent_conditions}
Let $A:H \to H'$ be a bounded linear operator between harmonic $\Lie{K}$-spaces.  The following are equivalent:
\begin{enumerate}
\item  $A \in \scrA[S]$,
\item  For any $\sigma\in\Khat_S$, and any $\epsilon>0$, there is a finite set $\irrepset\subset\Khat_S$ of $\Lie{K}_S$-types such that $\|p_\irrepset^\perp A p_\sigma\|<\epsilon$ and $\| p_\sigma A p_\irrepset^\perp \| < \epsilon$.
\item For any $\sigma\in\Khat_S$, $Ap_\sigma$ and $p_\sigma A$ are in $\scrK[S]$.
\item  $A$ is a two-sided multiplier of $\scrK[S]$, meaning that $AK\in\scrK[S]$ for all right-composable $K\in\scrK[S]$, and $KA\in\scrK[S]$ for all left-composable $K\in \scrK[S]$.
\end{enumerate}
\end{proposition}

We now describe some considerable simplifications from \cite{Yuncken:PsiDOs} in the case of homogeneous vector bundles for $\SU(3)$.

\begin{lemma}
\label{lem:degeneracy}
Let $E$, $E'$ be $\Lie{K}$-homogeneous vector bundles over $\scrX$, and put $H=\LXE$, $H'= \LXEprime$.  Then $\scrK[\Sigma](H,H') = \scrK(H,H')$ and $\scrA[\Sigma](H,H') = \scrK[\emptyset](H,H') = \scrA[\emptyset](H,H') = \scrL(H,H')$.

\end{lemma}

\begin{proof}
Since $H$ and $H'$ are direct sums of finitely many weight spaces for the right regular representation of $\Lie{K}$, any bounded operator from $H$ to $H'$ is $\Lie{M}$-harmonically finite.  Hence, $\scrK[\emptyset](H, H') = \scrA[\emptyset](H,H') = \scrL(H,H')$.

Lemma 3.3 of \cite{Yuncken:PsiDOs} shows that $\scrK[\Sigma](H,H')= \scrK(H,H')$.  By Proposition \ref{prop:AS_equivalent_conditions} above, any bounded operator $A:H\to H'$ is in $\scrA[\Sigma]$.
\end{proof}

The only nontrivial cases, then, are $\scrK[\{\alpha_i\}]$ and $\scrA[\{\alpha_i\}]$, which we abbreviate as $\scrK[\alpha_i]$ and $\scrA[\alpha_i]$.

\begin{definition}
\label{def:A}
As in \cite{Yuncken:PsiDOs}, we put $\scrA \defeq \cap_{S\subseteq\Sigma} \scrA[S]$, the simultaneous multiplier category of all $\scrK[S]$ ($S\subseteq\Sigma$). 
Note, though, that by Lemma \ref{lem:degeneracy} this reduces to $\scrA(H,H') = \scrA[\alpha_1](H,H') \cap \scrA[\alpha_2](H,H')$ when $H$, $H'$ are $L^2$-section spaces of homogeneous vector bundles.
\end{definition}

In the generality of \cite{Yuncken:PsiDOs}, it is necessary to adjust the operator spaces $\scrK[S]$ by defining $\scrJ[S] \defeq \scrK[S] \cap \scrA$.  The next lemma shows that this is not necessary for the current application.

\begin{lemma}
\label{lem:Ji_is_Ki}

With $H,H'$ as in Lemma \ref{lem:degeneracy}, $\scrK[\alpha_i](H,H') \subseteq \scrA(H,H')$, for $i=1,2$.
Thus, $\scrJ[\alpha_i](H,H') = \scrK[\alpha_i](H,H')$.

\end{lemma}

\begin{proof}
Let $i=1$.  
It is immediate that $\scrK[\alpha_1](H,H')\subseteq \scrA[\alpha_1](H,H')$.  Lemma 5.4 of \cite{Yuncken:PsiDOs} implies that on $H$ and $H'$, $p_{\sigma_1} p_{\sigma_2}$ is compact for any $\sigma_1\in\Khat_1$ and $\sigma_2\in\Khat_2$.  Thus, if $K:H\to H'$ is $\Lie{K}_1$-harmonically finite, then $K p_{\sigma_2}\in\scrK(H,H') \subseteq \scrK[\alpha_2](H,H')$.  By Proposition \ref{prop:AS_equivalent_conditions}, $K\in\scrA[\alpha_2](H,H')$.  Taking the norm-closure, $\scrK[\alpha_1](H,H')\subseteq \scrA[\alpha_2](H,H')$, which proves the result.  The case $i=2$ is analogous.
\end{proof}

We therefore avoid the notation $\scrJ[\alpha_i]$ altogether.

\begin{theorem}[\thmcitemore{Yuncken:PsiDOs}{Theorem 1.11}]
\label{thm:lattice_of_ideals}
Let $E$ be a $\Lie{K}$-homogeneous vector bundle over $\scrX$, and $H\defeq\LXE$.  Then
\begin{enumerate}
\item $\scrK[\alpha_i](H)$ is an ideal in $\scrA(H)$, for $i=1,2$.
\item $\scrK[\alpha_1](H) \cap \scrK[\alpha_2](H) = \scrK(H)$.
\end{enumerate}

\end{theorem}

\begin{lemma}[\thmcitemore{Yuncken:PsiDOs}{Lemma 8.1}]
\label{lem:mult_ops_in_A}
  Let $\mu$, $\nu$ be weights.  For any $f\in\CXE[\mu-\nu]$, the multiplication operator $\multop{f}:\LXE[\nu] \to \LXE[\mu]$ is in $\scrA$.
\end{lemma}

\begin{remark}
\label{rmk:mult_ops_in_compact_picture}
Lemma \ref{lem:mult_ops_in_A} depends on $\Lie{K}$-equivariant structure only, so that $f$ may be (the restriction to $\Lie{K}$ of) a section of $L_{(\mu-\nu)\oplus\chiA}$ for any $\chiA\in\lie{m}_\CC^\dual$.
\end{remark}


\subsection{Principal series representations}
\label{sec:principal_series}

The purpose of this section is to prove the following important fact, the first of two rather technical harmonic analysis results.

\begin{proposition}
\label{prop:U(g)_in_A}
Let $\mu\in\Weights$.  For any $g\in \Lie{G}$, $U_\mu(g) \in\scrA(\LXE[\mu])$.
\end{proposition}

We will use the notation for the elements of $\lie{k}_\CC$ from Section \ref{sec:Lie_groups}, noting that the elements $X_\alpha$, $Y_\alpha$ ($\alpha\in\Roots^+$) and $H_i$, $H_i'$ (for either $i=1$ or $2$) form a basis for $\lie{g}$.  We let $X_\alpha^\dual, Y_\alpha^\dual, H_i^\dual, {H_i'}^\dual$ denote the dual basis elements of $\lie{g}^\dual$.  We also recall the notation $\matrixunit{\eta^\dual}{\xi}$ for matrix units.

\begin{lemma}
\label{lem:a-action}
Let $A\in\lie{a}$.  Let $\pi\in\Khat$ and $\eta^\dual \in V^{\pi\dual}$, $\xi \in (V^\pi)_{-\mu}$.  Then $U_{\mu}(A)\matrixunit{\eta^\dual}{\xi} = \matrixunit{\eta^\dual\otimes A}{\Xi(\xi)}$, where
$$
  \Xi(\xi) \defeq \rho(H_i)\xi\otimes H_i^\dual + \rho(H_i')\xi \otimes {H_i'}^\dual 
    + \sum_{\alpha\in\Delta} \sign(\alpha) \pi(X_{\alpha})\xi\otimes X_\alpha^\dual 
    \quad \in V^\pi \otimes \lie{g}^\dual.
$$
\end{lemma}

Note that $\matrixunit{\eta^\dual\otimes A}{\Xi(\xi)}$ is a matrix unit for the \emph{non-irreducible} representation $\pi\otimes\Ad^\dual$, hence a sum of matrix units for the irreducible components of $\pi\otimes\Ad^\dual$.

\begin{proof}
Define functions $\kappa$, $\Lie{a}$, $\Lie{n}$ on $\Lie{G}$ using the Iwasawa decomposition:
$$
  g =: \kappa(g) \Lie{a}(g) \Lie{n}(g) \in \Lie{KAN}, \qquad \text{for $g\in\Lie{G}$}.
$$
The derivatives $\Derivative \kappa_e$, $\Derivative \Lie{a}_e$ and $\Derivative \Lie{n}_e$ at the identity are the ($\RR$-linear) projections of $\lie{g}$  onto the components of the decomposition $\lie{g}=\lie{k\oplus a\oplus n}$.  If $P\in\lie{g}$, let us write $P= P_+ +  P_0 +  P_-$ where $P_+$, $P_0$, $P_-$ are strictly upper-triangular, diagonal, and strictly lower-triangular, respectively.  If $P$ is self-adjoint, the $\lie{k\oplus a\oplus n}$ decomposition of $P$ is $P=(-P_+ + P_-) \oplus P_0 \oplus 2P_+$.  Thus,
\begin{eqnarray}
\label{eq:k-derivative}
  \Derivative \kappa_e (P) &=& \left(-\sum_{\alpha\in\Delta} \sign(\alpha) X_\alpha\otimes X_\alpha^\dual \right) P, \\
\label{eq:a-derivative}
  \Derivative \Lie{a}_e (P) &=& \big( H_i \otimes H_i^\dual + H_i'\otimes {H_i'}^\dual \big) P.
\end{eqnarray}

For $a\in\Lie{A}$, $k\in\Lie{K}$, 
$$
  a^{-1}k = kk^{-1}a^{-1}k = k\,\kappa(k^{-1}a^{-1}k) \Lie{a}(k^{-1}a^{-1}k) \Lie{n}(k^{-1}a^{-1}k).
$$
In order to describe the $\Lie{G}$-action on a $\Lie{K}$-matrix unit, one must extend $\matrixunit{\eta^\dual}{\xi}$ to a $\Lie{B}$-equivariant function on $\Lie{G}$.  Equation \eref{eq:B-equivariance}) gives
\begin{eqnarray}
  U_{\mu}(a)\matrixunit{\eta^\dual}{\xi}(k)
    &\defeq& \matrixunit{\eta^\dual}{\xi}(a^{-1} k) \nonumber\\
    &=& e^{\rho}(\Lie{a}(k^{-1}ak)) \matrixunit{\eta^\dual}{\xi}(k\, \kappa(k^{-1}a^{-1}k)) \nonumber\\
    &=& e^{\rho}(\Lie{a}(k^{-1}ak)) \big( \eta^\dual, \pi(k) \pi(\kappa(k^{-1}a^{-1}k)) \xi\big).
    \label{eq:extension_to_G}
\end{eqnarray}
Let $a=\exp(tA)$, and take the derivative with respect to $t$ at $t=0$:
\begin{equation*}
\label{eq:UA1}
  U_{\mu}(A)\matrixunit{\eta^\dual}{\xi}(k)
    = \rho(\Derivative \Lie{a}_e(\Ad k^{-1}(A))) \big( \eta^\dual, \pi(k) \xi\big)
      - \big( \eta^\dual, \pi(k) \pi(\Derivative \kappa_e(\Ad k^{-1}(A)) \xi\big).
\end{equation*}
Since $\Ad k^{-1} (A)$ is self-adjoint, Equations \eref{eq:k-derivative} and \eref{eq:a-derivative} give
\begin{eqnarray*}
  \lefteqn{U_{\mu}(A)\matrixunit{\eta^\dual}{\xi}(k) } \quad \\
    &=& \rho(H_i) \big(H_i^\dual, \Ad k^{-1}(A)\big) \big( \eta^\dual, \pi(k) \xi\big)
      + \rho(H_i') \big({H_i'}^\dual, \Ad k^{-1}(A)\big) \big( \eta^\dual, \pi(k) \xi\big) \\
    &&\qquad +\sum_{\alpha\in\Delta} \sign(\alpha) \big( \eta^\dual, \pi(k) \pi(X_\alpha) \big(X_\alpha^\dual, \Ad k^{-1}(A)\big) \xi\big) \\
    &=& \big( A, \Ad^\dual k (H_i^\dual) \big) \big( \eta^\dual, \pi(k) \rho(H_i)\xi\big)
      + \big( A, \Ad^\dual k ({H_i'}^\dual) \big) \big( \eta^\dual, \pi(k) \rho(H_i')\xi\big) \\
    && \qquad  +\sum_{\alpha\in\Delta} \sign(\alpha) \big(A,  \Ad^\dual k (X_\alpha^\dual)\big)\big( \eta^\dual, \pi(k) \pi(X_\alpha) \xi\big) \\
    &=& \matrixunit{\eta^\dual\otimes A}{\Xi(\xi)}(k).
\end{eqnarray*}

\end{proof}

Recall the decomposition $(\lie{k}_i)_\CC = \lie{s}_i \oplus \lie{z}_i$ of Section \ref{sec:parabolic_subgroups}.  Let $\mu\in\Weights$.  Since $\lie{z}_i\subseteq\lie{h}$, the action of $\lie{z}_i$ on the $(-\mu)$-weight space of any $\Lie{K}$-representation is completely determined by $\mu$.  Thus, the $\Lie{K}_i$-isotypical subspaces of $\LXE[\mu] $ are the $\lie{s}_i$-isotypical subspaces.  Moreover, since $\LXE[\mu]$ has $\lie{s}_i$-weight $-\mu_i \defeq -\mu(H_i)$, the $\Lie{s}_i$-types which occur must have highest weights $|\mu_i|, |\mu_i|+2, \ldots$

In what follows, we fix $i=1$ or $2$ and let $\sigma_l$ denote the $\lie{s}_i$-type with highest weight $l\in\NN$.  We abbreviate $p_l \defeq p_{\sigma_l}$.  Note that $p_l=0$ on $\LXE[\mu]$ if $l \not\equiv \mu_i \pmod{2}$ or $l < |\mu_i|$.
The next lemma shows that $U_\mu(A)$ is tridiagonal with respect to $\Lie{K}_i$-types, and that the off-diagonal entries have at most linear growth.

\begin{lemma}
\label{lem:tridiagonal}
Fix $\mu\in\Weights$ and let $A\in \lie{a}$.  There exists a constant $C>0$ such that for any $m,l \in \NN$,
\begin{equation*}
  \begin{array}{rcll}
  \| p_m U_\mu(A) p_l \|  &=&    0 &\text{if $|m-l|>2$}, \\
  \| p_m U_\mu(A) p_l \|  &\leq& C(l+1) \qquad &\text{if $|m-l|=2$}.
  \end{array}
\end{equation*}

\end{lemma}

\begin{proof}
Let us take $i=1$, with the case of $i=2$ being entirely analogous.
Suppose $\matrixunit{\eta^\dual}{\xi} \in p_l \LXE[\mu]$, which is to say that $\eta^\dual\in V^{\pi\dual}$, $\xi \in (V^\pi)_{\sigma_l}$ for some $\pi\in\Khat$.  By Lemma \ref{lem:a-action}, we need to understand the decomposition of $\Xi(\xi)$ into $\lie{s}_1$-types.

The adjoint representation of $\lie{g}$ decomposes into the $\lie{s}_1$-representations
$$
   \vspan{X_1, H_1, Y_1}, \quad \vspan{H_1'}, \quad \vspan{X_2, X_3}, \quad \vspan{Y_2, Y_3},
$$
and $\lie{g}^\dual$ decomposes dually.  We break up the expression for $\Xi(\xi)$ into corresponding parts.

Firstly, $H_1'$ has trivial $\lie{s}_1$-type, so $\rho(H_1')\xi\otimes{H_1'}^\dual$ has $\lie{s}_i$-type $l$.  Next, note that the vector $X_2 \otimes X_2^\dual + X_3 \otimes X_3^\dual \in \lie{g}\otimes\lie{g}^\dual$ also has trivial $\lie{s}_1$-type, since it corresponds to the identity map on the subrepresentation $\vspan{X_2, X_3}$.  
The map
\begin{eqnarray*}
  V^\pi \otimes \lie{g}\otimes \lie{g}^\dual &\to& V^\pi \otimes \lie{g}^\dual \\
  \zeta \otimes Z \otimes Z^\dual &\mapsto& \pi(Z)\zeta \otimes Z^\dual
\end{eqnarray*}
is a morphism of $\Lie{K}$-representations, in particular of $\Lie{s}_i$-representations, so $\pi(X_2)\xi \otimes X_2^\dual + \pi(X_3) \xi \otimes X_3^\dual$ also has $\lie{s}_1$-type $l$.  Similarly, $-\pi(Y_2)\xi \otimes Y_2^\dual - \pi(Y_3) \xi \otimes Y_3^\dual$ has $\lie{s}_1$-type $l$.

Thus, all the off-diagonal components of $U_\mu(A)$ are due to the components
\begin{equation}
\label{eq:off-diagonal_terms}
  \Xi_1(\xi) \defeq \rho(H_1)\xi \otimes H_1^\dual + \pi(X_1)\xi \otimes X_1^\dual - \pi(Y_1) \xi \otimes Y_1^\dual
\end{equation}
of $\Xi(\xi)$.  The coadjoint representation of $\lie{s}_1$ on $\vspan{X_1^\dual, H_1^\dual, Y_1^\dual}$ has highest weight $2$, so the fusion rules for $\SU(2)$-representations imply that \eref{eq:off-diagonal_terms} contains $\lie{s}_i$-types $l-2, l, l+2$ only.  

It remains to prove the norm estimate on the off-diagonal terms.  By Equations \eref{eq:X-formula}--\eref{eq:Y-formula},
\begin{equation*}
\begin{array}{rclcl}
  \|\rho(H_1) \xi\| &=& 2\|\xi\| &\leq& (l+1)\|\xi\|, \\ \\
  \|\pi(X_1) \xi \| &=& \frac{1}{2} \sqrt{(l-\mu_i)(l+\mu_i+2)} \|\xi\| &\leq& (l+1)\|\xi\|, \\ \\
  \|\pi(Y_1) \xi \| &=& \frac{1}{2} \sqrt{(l-\mu_i+2)(l+\mu_i)} \|\xi\| &\leq& (l+1)\|\xi\|,
\end{array}
\end{equation*}
so the norm of $\Xi_1(\xi)$ is bounded by $C_0(l+1) \|\xi\|$ for some constant $C_0$.
We need to convert this into a bound on the norm of the matrix units.  

Decompose $\pi\otimes \Ad^\dual$ into irreducible $\Lie{K}$-subrepresentations.  Suppose $\pi'$ is an irreducible subrepresentation of $\pi\otimes\Ad^\dual$.  By orthogonality of characters, $\pi$ is a subrepresentation of $\pi'\otimes \Ad$.  Therefore $\dim \pi \leq \dim (\pi'\otimes\Ad) = 8\dim \pi' $, so that $\dim \pi' \geq \frac{1}{8} \dim \pi$.  This also shows that the number of irreducible components of $\pi\otimes\Ad^\dual$ is at most $64$.  

For each irreducible subrepresentation $\pi'$ of $\pi\otimes\Ad^\dual$, let  $y_{\pi'}^\dual$ denote the ${\pi'}^\dual$-component of $\eta^\dual\otimes A$,  and $x_{\pi'}$ the $\pi'$-component of $\Xi_1(\xi)$.
We get
\begin{eqnarray*}
  \| p_{l\pm2} U_\mu(A) p_l \matrixunit{\eta^\dual}{\xi} \|^2
    & \leq & \| \matrixunit{\eta\otimes A}{\Xi_1(\xi)} \|^2 \\
    & = & \sum_{\pi'} \frac{1}{\dim \pi'} \|y_{\pi'}^\dual\|^2   \|x_{\pi'}\|^2 \\
    &\leq& \sum_{\pi'} \frac{1}{\dim \pi'} \| \eta^\dual \otimes A \|^2 \| \Xi_1(\xi) \|^2 \\
    &\leq& \sum_{\pi'} \frac{1}{\dim \pi'} \|\eta^\dual\|^2 \|A\|^2 C_0^2(l+1)^2 \|\xi\|^2 \\
    &\leq& \|A\|^2 C_0^2 (l+1)^2   \sum_{\pi'} \frac{8}{\dim \pi} \| \eta^\dual\|^2 \| \xi \|^2 \\
    &\leq& 8.64.\|A\|^2 C_0^2 (l+1)^2 \| \matrixunit{\eta^\dual}{\xi} \|^2.
\end{eqnarray*}
Putting $C=  \sqrt{512}\, \|A\| \,C_0$ gives the result.

\end{proof}

\begin{proof}[Proof of Proposition \ref{prop:U(g)_in_A}]
We need to show $U_\mu(g)\in\scrA[\alpha_i]$ for $i=1,2$.
For $k\in\Lie{K}$, the left translation action $U_\mu(k)$ commutes with the decomposition into right $\Lie{K}_i$-types, so that $U_\mu(k) \in \scrA[\alpha_i]$ trivially.  By the $\Lie{KAK}$-decomposition, it suffices to prove the proposition for $g=a\in\Lie{A}$.

We continue with the notation of the previous lemma.  Put $P_m\defeq \sum_{j=0}^m p_j$.  We will show that for any $l\in\NN$ and any $\epsilon>0$, there exists $m\in\NN$ such that
$\| P_m^\perp U_\mu(a) p_l\| <\epsilon$ and $\|  p_l U_\mu(a) P_m^\perp\| <\epsilon$, from which Lemma \ref{prop:AS_equivalent_conditions} gives $U_\mu(a)\in\scrA[\alpha_i]$.

Let $A\in\lie{a}$ such that $e^A = a$.   Define $\phi:\NN\to [0,1]$ by
$$
  \phi(n) \defeq \begin{cases}
    1, & n\leq l, \\
    \max\,\{0, 1-\frac{\epsilon^2}{4C} \log(n+3) \}, & n>l,
  \end{cases}
$$
where $C$ is the constant of the previous lemma.  Define  $\Phi\defeq \sum_{n\in\NN} \phi(n) p_n$, an operator on $\LXE[\mu]$ which is scalar on each $\Lie{K}_i$-type.

We now decompose $U_\mu(A)$ into its diagonal and off-diagonal components.  For convenience of notation, we put $U\defeq U_\mu(A)$, then write $U=U_-+U_0+U_+$, where
$$
  U_- = \sum_{n=2}^\infty p_{n-2}U p_n, \qquad
  U_0 = \sum_{n=0}^\infty p_n U p_n, \qquad
  U_+ = \sum_{n=0}^\infty p_{n+2} U p_n.
$$
The diagonal component $U_0$ commutes with $\Phi$.  On the other hand,
\begin{eqnarray*}
  \| [p_{n-2} U p_n, \Phi] \| 
    &=& \| (\phi(n) - \phi(n-2))\, p_{n-2} U p_n \| \\
    &\leq& \frac{\epsilon^2}{4C}(\log(n+3)-\log(n+1)) \\
    &\leq& \frac{\epsilon^2}{2C} \frac{1}{(n+1)} \\
    &\leq& \frac{\epsilon^2}{2},
\end{eqnarray*}
by Lemma \ref{lem:tridiagonal}.  Thus,
$$
  \| [U_-, \Phi] \| = \sup_{n\in\NN} \| [U_{n-2,n}, \Phi] \|  \leq \frac{1}{2} \epsilon^2.
$$
Similarly, $\| [U_+, \Phi] \| \leq \frac{1}{2} \epsilon^2$.  Therefore, $\| [U_\mu(A), \Phi] \| \leq \epsilon^2$.

Let $s \in p_l \LXE[\mu]$ have norm one.  Put $s_t \defeq U_\mu(e^{tA})s$ for $0\leq t \leq 1$.  Then
$$
  | \frac{d}{dt} \ip{ \Phi s_t, s_t } |
    = | \ip{ \Phi U_\mu(A) s_t, s_t } + \ip{ \Phi s_t, U_\mu(A) s_t } |
    = | \ip{ [\Phi, U_\mu(A)] s_t, s_t } |
    \leq \epsilon^2,
$$
for all $t$.  Therefore,
\begin{eqnarray*}
  |\ip{\Phi s_1, s_1}| &=& \left| \ip{\Phi s_0, s_0} + \int_{t'=0}^1 \frac{d}{dt} \ip{\Phi s_t, s_t} \, dt'  \right|\\
    &\geq& 1 - \epsilon^2. 
\end{eqnarray*}
Let $m$ be the smallest integer for which $\phi(m)=0$.   Put $v \defeq P_m s_1$ and $w\defeq P_m^\perp s_1$.   Then $\|v\|^2 + \|w\|^2 =1$, but also
$$
  \| v\|^2  > \ip{\Phi v,v} =  \ip{\Phi v,v} + \ip{\Phi w,w} = \ip{\Phi s_1, s_1} \geq 1 -\epsilon^2.
$$
It follows that $ \|w\| < \epsilon$, {\em ie}, $\|P_m^\perp U_\mu(a) s \| < \epsilon$.  Since $s\in p_l\LXE[\mu]$ was arbitrary, $\|P_m^\perp U_\mu(a) p_l\| < \epsilon$.  

Replacing $a$ with $a^{-1}$, there exists $m'\in\NN$ such that $\| P_{m'}^\perp U_\mu(a^{-1}) p_l \| < \epsilon$.  Thus, after enlarging $m$ to be at least $m'$, we have
$$
  \| p_l U_\mu(a) P_{m}^\perp \| = \| P_{m}^\perp U_\mu(a^{-1}) p_l \| < \epsilon.
$$

\end{proof}

In fact, Proposition \ref{prop:U(g)_in_A} holds for any generalized principal series representation.  Although we don't actually need this here, it is now trivial to prove.

\begin{corollary}
For any $\Lie{G}$-homogeneous line bundle $\LXL[\chi]$ over $\scrX$, the translation operators $\sect \mapsto g\cdot \sect$ belong to $\scrA$.
\end{corollary}

\begin{proof}
Let $\chi=\chiM\oplus\chiA$.  A computation of the form of Eq.~\eref{eq:extension_to_G}
gives 
$$
  g\cdot\sect(k) = e^{\chiA}(\Lie{a}(k^{-1}gk)) \sect (k\, \kappa(k^{-1}g^{-1}k)),
$$
for any $k\in\Lie{K}$, while
$$
  U_{\chiM}(g)\sect(k) = e^{\rho}(\Lie{a}(k^{-1}gk)) \sect (k\, \kappa(k^{-1}g^{-1}k)).
$$
Note that $\Lie{a}(m^{-1}gm) = \Lie{a}(g)$ for any $m\in\Lie{M}$, $g\in\Lie{G}$.
Therefore, $g\cdot\sect = \multop{f} U_{\chiM}(g) \sect$, where $f(k) \defeq e^{\chiA-\rho}(\Lie{a}(k^{-1}gk))$ is in $C(\Lie{K/M}) = C(\scrX)$.  Since $\multop{f}$ and $U_{\chiM}(g)$ are in $\scrA$, we are done.

\end{proof}

\subsection{Longitudinal pseudodifferential operators}
\label{sec:PsiDOs}

Let $X\in\lie{k}_\CC$ be a root vector, of weight $\alpha$.  Via the right regular representation, $X$ defines a left $\Lie{K}$-invariant differential operator on $C^\infty(\Lie{K})$.  
For each weight $\mu$, $X$ maps $p_{-\mu} L^2(\Lie{K})$ to $p_{-\mu+\alpha} L^2(\Lie{K})$, so it defines a $\Lie{K}$-invariant differential operator
$$
  X : \LXE[\mu] \to \LXE[\mu-\alpha].
$$
The principal symbol of this differential operator is a $\Lie{K}$-equivariant linear map from the cotangent bundle $\Tangent^*\scrX \cong K\times_M (\lie{k/m})^*$ to $\End( E_\mu, E_{\mu-\alpha}) \cong E_{-\alpha}$.  (Here $(\lie{k/m})^*$ denotes the {\em real} dual of $\lie{k/m}$.)  By equivariance, this map is determined by its value on the cotangent fibre at the identity coset $\coset{e}\in\scrX$, which is
\begin{eqnarray}
\label{eq:diff_op_symbol}
  \Symbol (X) :  \Tangent^*_{\coset{e}}\scrX = (\lie{k/m})^*  
    & \to & \CC \\
  \xi & \mapsto & \xi(X). \nonumber
\end{eqnarray}


If $X\in(\lie{k}_i)_\CC$ ($i=1$ or $2$), then the differential operator $X:\CinftyXE[\mu] \to \CinftyXE[\mu-\alpha]$ is tangential to the foliation $\foliation[i]$ of Section \ref{sec:parabolic_subgroups}.  We will refer to such an operator as an $\foliation[i]$-longitudinal differential operator.  Its {longitudinal principal symbol} is the $\Lie{K}$-equivariant map $\Symbol[i]: \foliation[i]^* \to E_{-\alpha}$ which, at the identity coset, is given by
\begin{eqnarray*}
  \Symbol[i] :  (\foliation[i]^*)_{\coset{e}}  = (\lie{k}_i / \lie{m})^* &\to& \CC \\
  \xi &\mapsto & \xi(X).
\end{eqnarray*}

An $\foliation[i]$-longitudinal differential operator is \emph{longitudinally elliptic} if its longitudinal principal symbol is invertible off the zero section of $\Tangent^*\foliation[i]$.  Note that $X_i = -\half(X_i' + \sqrt{-1} \, X_i'') \in (\lie{k}_i)_\CC$ where
$$
  X_i' = \smatrix{ 0 & -1 \\ 1 & 0 }, \quad X_i'' = \smatrix{ 0 & \sqrt{-1} \\ \sqrt{-1} & 0 }
$$
span $\lie{k}_i/\lie{m}$, so that $X_i$ is $\foliation[i]$-longitudinally elliptic.  Similarly, $Y_i$ is $\foliation[i]$-longitudinally elliptic.  Moreover, $X_i$ and $Y_i$ are formal adjoints.  We shall use $X_i$, $Y_i$ also to denote their closures as unbounded operators on the $L^2$-section spaces.

Fix $\mu\in\Weights$.  Let $E \defeq E_\mu \oplus E_{\mu-\alpha_i}$, and define  $D_i \defeq \smatrix{0&Y_i\\X_i&0}$ on $\LXE$.  The $\lie{s}_i$-isotypical subspaces of $\LXE$ are eigenspaces for $D_i$, and by the representation theory of $\lie{s}_i$---specifically Equations \eref{eq:X-formula} and \eref{eq:Y-formula}---its spectrum is discrete.

\medskip

For the definition and basic properties of longitudinal pseudodifferential operators, we refer the reader to \cite{MS-GAFS}\footnote{In this reference, they are called {tangential} pseudodifferential operators.}.  If $E$, $E'$ are vector bundles over $\scrX$, we denote the set of $\foliation[i]$-longitudinal pseudodifferential operators of order at most $p$ by $\PsiDO[i]^p(E,E')$.  If $E=E'$, we abbreviate this to $\PsiDO[i]^p(E)$.

Let $C(\cosphere\foliation[i]; \End(E))$ denote the algebra of continuous sections of the pullback of $\End(E)$ to the cosphere bundle of the foliation $\foliation[i]$.  The longitudinal principal symbol map $\Symbol[i]:\PsiDO[i]^0(E) \to C(\cosphere\foliation[i]; \End(E))$ extends to the operator-norm closure $\overline{\PsiDO[i]^{0}}(E)$, and we have Connes' short exact sequence,
\begin{equation}
\label{eq:symbol-sequence}
  \xymatrix{
    0 \ar[r] &
    \overline{\PsiDO[i]^{-1}}(E) \ar[r] &
    \overline{\PsiDO[i]^{0}}(E) \ar[r]^-{\Symbol[i]} &
    C(\cosphere\foliation[i]; \End(E)) \ar[r] &
    0.
  }
\end{equation}

For any closed, densely defined, unbounded operator $T$ between Hilbert spaces, we let $\Phase T$ denote the phase in the polar decomposition: $T = (\Phase{T}) |T|$.  We also use $\Phase{z}$ to denote the phase of a complex number $z\in\CC^\times$.

\begin{lemma}
\label{lem:F_in_PsiDO}
For any weight $\mu$, $\Phase{X_i}:\LXE[\mu] \to \LXE[\mu-\alpha_i]$ and $\Phase{Y_i}:\LXE[\mu-\alpha_i] \to \LXE[\mu]$ are $\foliation[i]$-longitudinal pseudodifferential operators.  Their longitudinal principal symbols at the identity coset are
\begin{eqnarray*}
  \Symbol[i] (\Phase{X_i}) (\xi) &=& \Phase{(\xi(X_i))}, \\
  \Symbol[i] (\Phase{Y_i}) (\xi) &=& \Phase{(\xi(Y_i))} \;=\; \overline{\Phase{(\xi(X_i))}}.
\end{eqnarray*}
for $ \xi$ in the unit sphere of $(\lie{k}_i/\lie{m})^* \cong (\foliation[i]^*)_{\coset{e}}$.
\end{lemma}

\begin{proof}
Let $E\defeq E_\mu\oplus E_{\mu-\alpha_i}$.
Fix $\epsilon>0$ such that $\Spec(D_i) \cap (-\epsilon,\epsilon) = \{0\}$.  Let $f:\RR\to[-1,1]$ be smooth with $f(0)=0$ and $f(x) = \sign(x)$ for all $|x|\geq\epsilon$.  
A fibrewise application of \cite[Theorem 1.3]{Taylor} shows that $ f(D_i) = \Phase{D_i} \in \PsiDO[i]^0(\scrX;E)$.  Moreover the proof of the theorem shows that its full symbol has an asymptotic expansion with leading term $f(\Symbol[i] D_i)$.  Note that
$$
  (\Symbol[i] D_i)(\xi) = \smatrix{ 0 & \xi(X_i) \\ \overline{\xi(X_i)} &0 }
$$
has spectrum $\{\pm | \xi(X_i) | \}$, so if $\xi$ is large enough that $|\xi(X_i)|>\epsilon$, then
$$
  f (\Symbol[i] D_i)(\xi) = \Phase{(\Symbol[i] D_i(\xi))} 
    = \smatrix{ 0 & \Phase{(\xi(X_i))} \\ \Phase{(\overline{\xi(X_i)})} &0 }.
$$
This is radially constant on $(\lie{k}_i/\lie{m})^*$ for $|\xi(X_i)|>\epsilon$.  The principal symbol is the limit at the sphere at infinity.  
\end{proof}

\begin{theorem}
\label{prop:PsiDOs_in_K}
\label{thm:PsiDOs_in_A}
Let $E, E'$ be $\Lie{K}$-homogeneous vector bundles over $\scrX$.  Then 
\begin{enumerate}
\item $\PsiDO[i]^{-1}(E,E') \subseteq \scrK[\alpha_i]$,
\item $\PsiDO[i]^{0}(E,E') \subseteq \scrA$,
\end{enumerate}
\end{theorem}

Part (i) is proven in Proposition 1.12 of \cite{Yuncken:PsiDOs}.  It is also shown there that $\PsiDO[i]^0(E,E') \subseteq \scrA[i]$.  The more difficult question of showing $\PsiDO[i]^0(E,E') \subseteq \scrA[j]$ for $j\neq i$ requires some lengthy computations in noncommutative harmonic analysis.  In order not to disrupt the flow of ideas too severely, we have presented the proof in Appendix \ref{sec:PsiDOs_in_A}.

As an indication of the subtleties involved, we remark that the longitudinally elliptic differential operator $X_1$ is {\em not} an unbounded multiplier of $\scrK[\alpha_2]$.  To see this, note that $(1+X_1^*X_1)^{-\half} \in \PsiDO[1]^{-1}(E_\mu) \subseteq \scrK[\alpha_1]$.  Since $\scrK[\alpha_1].\scrK[\alpha_2] \subseteq \scrK$, the range of $(1+X_1^*X_1)^{-\half}$ as a multiplier of $\scrK[\alpha_2]$ is not dense.  Thus, $X_1$ is not regular with respect to $\scrK[\alpha_2]$ (see \cite[Chapter 10]{Lance}).  Hence, proving that $\Phase{X_1}$ multiplies $\scrK[2]$ can not be achieved by direct functional calculus.

\begin{lemma}
\label{lem:F_f_commute}
Let $i=1,2$ and let $\mu, \nu$ be weights.  For any $f\in\CXE[\nu-\mu]$, the diagram
$$ 
  \xymatrix{
    \LXE[\mu] \ar[r]^{M_f} \ar[d]_{\Phase{X_i}} &
      \LXE[\nu] \ar[d]^{\Phase{X_i}} \\
    \LXE[\mu-\alpha_i] \ar[r]_{M_f} &
      \LXE[\nu-\alpha_i]
  }
$$
commutes modulo $\scrK[\alpha_i]$.
\end{lemma}

\begin{remark}
\label{rem:commutator_notation}
We abbreviate this result by writing $[\Phase{X_i}, M_s] \in \scrK[\alpha_i]$.  By taking adjoints, we also have $[\Phase{Y_i}, M_s] \in \scrK[\alpha_i]$. 
\end{remark}

\begin{proof}
As an element of $C(\cosphere\foliation[i];E_{\alpha_i})$, the principal symbol of $\Phase{X_i}:\LXE[\mu]\to\LXE[\mu-\alpha_i]$ is independent of the weight $\mu$.  Thus, the above diagram commutes at the level of principal symbols.
\end{proof}


\section{The normalized BGG complex}
\label{sec:construction}

\subsection{$\Lie{G}$-continuity}
\label{sec:G-continuity}

Before embarking on the main construction, we need to make some remarks regarding the issue of $\Lie{G}$-continuity.  Recall that a bounded operator $A$ between unitary $\Lie{G}$-representations is $\Lie{G}$-continuous if the map $g \mapsto g .A. g^{-1}$ is continuous in the operator-norm topology. 

Rather than burden the notation with extra decorations, we choose to make the convention that {\bf throughout this section, we use $\scrK[\alpha_i]$ ($i=1,2$) to denote its $C^*$-subcategory of $\Lie{G}$-continuous elements.} 

This is reasonable, since almost every operator we deal with is $\Lie{G}$-continuous.  From \cite{AS4}, we know that for any homogeneous vector bundles $E$, $E'$ over $\scrX$, the set of longitudinal pseudodifferential operators $\overline{\PsiDO[i]^0}(E,E,')$ consists of $\Lie{G}$-continuous operators.  This includes continuous multiplication operators, in the sense of Section \ref{sec:homogeneous_vector_bundles} (which are $\Lie{G}$-continuous for much simpler reasons).
The notable exceptions, of course, are the representations $U_\mu(g)$ of the group elements themselves.

In the majority of instances, where $\Lie{G}$-continuity is a trivial consequence of the above remarks, we will not make specific mention of it in the proofs.

\subsection{Intertwining operators}

Let $\mu$, $\mu'$ be weights for $\Lie{K}= \SU(3)$.  It is well known that the principal series representations $U_\mu$ and $U_{\mu'}$ are unitarily equivalent if and only if $\mu'=w\cdot\mu$ for some Weyl group element $w\in W$.  
When $w=\reflection{\alpha_i}$ is a simple reflection corresponding to the root $\alpha_i$, there is a very concise formula for the intertwining operator.  

\begin{proposition}
\label{prop:intertwiner_formula}
Let $\mu$, $\mu'$ be weights with $\mu'=\reflection{\alpha_i}\mu$, so that $\mu-\mu' = n\alpha_i$ for some $n\in\ZZ$.  If $n>0$, the operator $(\Phase{X_i})^n:  \LXE[\mu] \to \LXE[\mu']$ intertwines $U_\mu$ and $U_\mu'$.  If $n<0$, then $(\Phase{Y_i})^n: \LXE[\mu] \to \LXE[\mu']$ is an intertwiner.
\end{proposition}

This is essentially the formula given by Duflo in \cite[Ch.~III]{Duflo}.  However, Duflo's formulation is sufficiently different that we feel a brief comparison is worthwhile.  

\begin{proof}
We follow the notation for $\slx(2,\CC)$-representations from the end of Section \ref{sec:parabolic_subgroups}.  Note that Equations \eref{eq:X-formula} and \eref{eq:Y-formula} imply that
$(\Phase{X})e_j = e_{j+2}$ and $(\Phase{Y}) e_j = e_{j-2}$.  Secondly, with $w=\smatrix{0&-1\\1&0}$, 
\begin{equation}
\label{eq:Phase(X)_vs_w}
\begin{array}{rcccl}
  (\Phase{X})^j \cdot e_{-j} &=& e_j& =& (-1)^{\half(\delta+j)} w\cdot e_{-j}, \\
  (\Phase{Y})^j \cdot e_{j} &= &e_{-j}& =& (-1)^{\half(\delta-j)} w\cdot e_{j},
\end{array}
\end{equation}
for any $j\geq 0$. (See \cite[\S{}III.3.5]{Duflo}.)

Recall that the restriction of $\mu$ to a weight of $\lie{s}_i$ is $\mu_i \defeq \mu(H_i)\in\ZZ$.  The hypotheses of the proposition are equivalent to saying $\mu_i = -\mu'_i = n$.

First consider the case $n>0$.  Let $A=A(w_i, \mu,0):\LXE[\mu] \to \LXE[\mu']$ be the intertwiner of \cite[\S{}III.3.1]{Duflo}.  The action of $A$ upon matrix units is given in \cite[\S{}III.3.3 and \S{}III.3.9]{Duflo} as follows.  Let $\pi\in\Khat$, $\eta^\dual\in V^{\pi\dual}$, $\xi\in p_{-\mu}(V^\pi)$ and suppose that $\xi$ lies in an irreducible $\lie{s}_i$-subrepresentation of $V^\pi$ with highest weight $\delta$.  Then, in the notation of Section \ref{sec:harmonic_notation}, $A:\matrixunit{\eta^\dual}{\xi} \mapsto \matrixunit{\eta^\dual}{\xi'}$ where
\begin{eqnarray*}
\label{eq:Duflo-formula}
  \xi' &=& (-1)^{\half(\delta + |\mu_i|)} |\mu_i|^{-1} \pi(w_i)\xi \\
    &=& |\mu_i|^{-1} (\Phase{X_i})^n \xi.
\end{eqnarray*}
Hence, $A=|\mu_i|^{-1} (\Phase{X_i})^n:\LXE[\mu] \to \LXE[\mu']$, where $X_i$ here denotes the right regular action.  Thus, $(\Phase{X_i})^n$ differs from $A$ by the positive scalar $|\mu_i|=n$.

The case $n<0$ follows since $\Phase{Y_i} = \Phase{X_i}^*$.

\end{proof}

\medskip

We now recap the directed graph structure which underlies the BGG complex.  For our $\K$-homological purposes, it will be convenient to make an undirected graph, or more accurately, to include also the reversal of each edge.  

As before, if $\alpha$ is a positive root, we use $\reflection{\alpha}\in\Weylgroup$ to denote the reflection in the wall orthogonal to $\alpha$.  For $w,w'\in\Weylgroup$, we write $\edge[\alpha]{w}{w'}$ if $w'=\reflection{\alpha}w$ and $\length(w') = \length(w) \pm 1$.  We will write $\edge{w}{w'}$ if $\edge[\alpha]{w}{w'}$ for some $\alpha\in\Roots^+$.  An edge $\edge[\alpha]{w}{w'}$ will be called  {\em simple} if $\alpha$ is a simple root.

For $\Lie{G}=\SL(3,\CC)$, this yields the graph
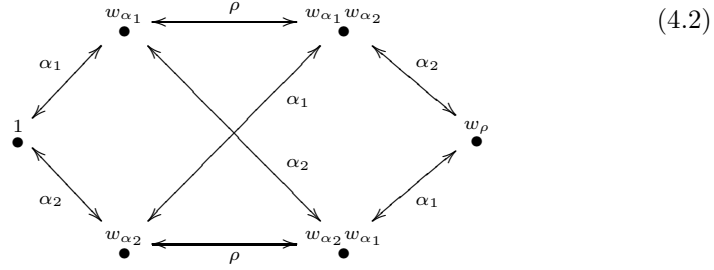
\begin{equation}
\label{eq:Weyl_graph}
  \xymatrix{
    & \stackrel{\reflection{\alpha_1}}{\bullet} \ar@{<->}[rr]^{\rho} \ar@{<->}[ddrr]^(0.7){\alpha_2} 
      && \stackrel{\reflection{\alpha_1}\reflection{\alpha_2}}{\bullet} \ar@{<->}[dr]^{\alpha_2} \\
    \stackrel{1}{\bullet} \ar@{<->}[ur]^{\alpha_1} \ar@{<->}[dr]_{\alpha_2} 
      &&&& \stackrel{w_\rho}{\bullet} \\
    & \stackrel{\reflection{\alpha_2}}{\bullet} \ar@{<->}[rr]_{\rho} \ar@{<->}[uurr]_(0.7){\alpha_1} 
      && \stackrel{\reflection{\alpha_2}\reflection{\alpha_1}}{\bullet} \ar@{<->}[ur]_{\alpha_1} 
  }
\end{equation}

\begin{definition}
\label{def:intertwiners}
Fix a dominant weight $\lambda$.  If $\edge[\alpha_i]{w}{w'}$ is a simple edge, we denote by $\intertwiner{\lambda}{w}{w'}$ the intertwining operator of Lemma \ref{prop:intertwiner_formula}:
$$
  \intertwiner{\lambda}{w}{w'} \defeq 
    \begin{cases} 
      (\Phase{X_i})^n  & \text{if $n\geq0$},\\
      (\Phase{Y_i})^{-n}  & \text{if $n\leq0$},
    \end{cases}      
$$
where $w\lambda - w'\lambda = n\alpha_i$.
These will be referred to as {\em simple intertwiners}.  Note that $\intertwiner{\lambda}{w'}{w} = \intertwiner{\lambda}{w}{w'}^*$.

For the non-simple edges, we define intertwiners as compositions of simple intertwiners:
\begin{eqnarray}
  \intertwiner{\lambda}{\reflection{\alpha_1}}{\reflection{\alpha_1}\reflection{\alpha_2}} &\defeq& 
    \intertwiner{\lambda}{\reflection{\alpha_2}}{\reflection{\alpha_1}\reflection{\alpha_2}}.
    \intertwiner{\lambda}{1}{\reflection{\alpha_2}}.
    \intertwiner{\lambda}{\reflection{\alpha_1}}{1} \nonumber \\
  \intertwiner{\lambda}{\reflection{\alpha_2}}{\reflection{\alpha_2}\reflection{\alpha_1}} &\defeq& 
    \intertwiner{\lambda}{\reflection{\alpha_1}}{\reflection{\alpha_2}\reflection{\alpha_1}}.
    \intertwiner{\lambda}{1}{\reflection{\alpha_1}}.
    \intertwiner{\lambda}{\reflection{\alpha_2}}{1}, 
  \label{eq:non-simple_intertwiners}
\end{eqnarray}
and 
$\intertwiner{\lambda}{\reflection{\alpha_1}\reflection{\alpha_2}}{\reflection{\alpha_1}} \defeq \intertwiner{\lambda}{\reflection{\alpha_1}}{\reflection{\alpha_1}\reflection{\alpha_2}}^*$, $\intertwiner{\lambda}{\reflection{\alpha_2}\reflection{\alpha_1}}{\reflection{\alpha_2}} \defeq \intertwiner{\lambda}{\reflection{\alpha_2}}{\reflection{\alpha_2}\reflection{\alpha_1}}^*$.
\end{definition}

\begin{remark}
\label{rmk:commuting_diagram_of_intertwiners}
Duflo's intertwiners form a commuting diagram of the form
\begin{equation}
\label{eq:simple_intertwiners}
  \xymatrix@!C=5ex{
    & \LXE[\reflection{\alpha_1}\lambda] \ar[ddrr] && \LXE[\reflection{\alpha_1}\reflection{\alpha_2}\lambda] \ar[dr] \\
   \LXE[\lambda] \ar[ur] \ar[dr] &&&& \LXE[w_\rho\lambda] \\
    & \LXE[\reflection{\alpha_2}\lambda] \ar[uurr] && \LXE[\reflection{\alpha_2}\reflection{\alpha_1}\lambda] \ar[ur] 
  }
\end{equation}
Since the simple intertwiners $\intertwiner{\lambda}{w}{w'}$ defined here are positive scalar multiples of Duflo's, 
the corresponding diagram of intertwiners $\intertwiner{\lambda}{w}{w'}$ commutes up to some positive scalar. But $\intertwiner{\lambda}{w}{w'} = (\Phase{X_i})^n$ is unitary, so that scalar is $1$.  The non-simple intertwiners defined by Equation \eref{eq:non-simple_intertwiners} are precisely those that complete \eref{eq:simple_intertwiners} to a commuting diagram of the form \eref{eq:Weyl_graph}.
\end{remark}

\begin{definition}
Define $\scrK[\rho] \defeq \scrK[\alpha_1] + \scrK[\alpha_2]$.  That is, $\scrK[\rho](H,H') \defeq \scrK[\alpha_1] (H,H') + \scrK[\alpha_2] (H,H') $ for any harmonic $\Lie{K}$-spaces $H$, $H'$.  Following the convention of Section \ref{sec:G-continuity}, we are including the condition of $\Lie{G}$-continuity in this definition.

\end{definition}

\begin{lemma}
\label{lem:intertwiners_in_A}
Let $\lambda$ be a dominant weight.
\begin{enumerate}
\item  For each $\edge{w}{w'}$, $\intertwiner{\lambda}{w}{w'} \in \scrA$.
\item  If $\edge[\alpha]{w}{w'}$, then $[\intertwiner{\lambda}{w}{w'}, \multop{f}] \in \scrK[\alpha]$ for any $f\in C(\scrX)$.
\end{enumerate}
\end{lemma}

\begin{proof}
Part {(i)} is immediate from Theorem \ref{thm:PsiDOs_in_A}.  If $\alpha$ is a simple root, then {(ii)} follows from Lemma \ref{lem:F_f_commute}.  For $\alpha=\rho$, there are four intertwiners to be checked.  The following calculation is representative of all of them:
\begin{eqnarray*}
[\intertwiner{\lambda}{\reflection{\alpha_1}}{\reflection{\alpha_1}\reflection{\alpha_2}} , \multop{f}] &=& 
  [\intertwiner{\lambda}{\reflection{\alpha_2}}{\reflection{\alpha_1}\reflection{\alpha_2}}, \multop{f}].
    \intertwiner{\lambda}{1}{\reflection{\alpha_2}}.
    \intertwiner{\lambda}{\reflection{\alpha_1}}{1} \\
  &&\quad + \intertwiner{\lambda}{\reflection{\alpha_2}}{\reflection{\alpha_1}\reflection{\alpha_2}}.
   [ \intertwiner{\lambda}{1}{\reflection{\alpha_2}}, \multop{f} ].
    \intertwiner{\lambda}{\reflection{\alpha_1}}{1} \\
  &&\quad +  \intertwiner{\lambda}{\reflection{\alpha_2}}{\reflection{\alpha_1}\reflection{\alpha_2}}.
    \intertwiner{\lambda}{1}{\reflection{\alpha_2}}.
    [ \intertwiner{\lambda}{\reflection{\alpha_1}}{1}, \multop{f} ] . \\
  &\in& \scrK[\alpha_1] + \scrK[\alpha_2] + \scrK[\alpha_1] = \scrK[\rho].
\end{eqnarray*}
\end{proof}

\subsection{Normalized BGG operators}

\begin{definition}
\label{def:shifted_action}
Define the {\em shifted action} of the Weyl group on weights by $w\shiftedaction \mu \defeq w(\mu+\rho) -\rho$.
\end{definition}

From now on, $\lambda$ will denote a dominant weight.

\begin{definition}
\label{def:BGG_operators}
If $\edge[\alpha_i]{w}{w'}$ is a simple edge, then $w\shiftedaction \lambda - w'\shiftedaction \lambda = n\alpha_i$ for some $n\in\ZZ$.  We define the {\em normalized BGG operator} $\BGG{\lambda}{w}{w'} : \LXE[w\shiftedaction \lambda] \to \LXE[w'\shiftedaction \lambda]$ by
$$
  \BGG{\lambda}{w}{w'} \defeq 
    \begin{cases} 
      (\Phase{X_i})^n  & \text{if $n\geq0$},\\
      (\Phase{Y_i})^{-n}  & \text{if $n\leq0$}.
    \end{cases}      
$$
where $w\shiftedaction\lambda-w'\shiftedaction\lambda = n \alpha_i$.

For the non-simple arrows, define
\begin{eqnarray*}
  \BGG{\lambda}{\reflection{\alpha_1}}{\reflection{\alpha_1}\reflection{\alpha_2}} &\defeq& 
    \BGG{\lambda}{\reflection{\alpha_2}}{\reflection{\alpha_1}\reflection{\alpha_2}}.
    \BGG{\lambda}{1}{\reflection{\alpha_2}}.
    \BGG{\lambda}{\reflection{\alpha_1}}{1} \\
  \BGG{\lambda}{\reflection{\alpha_2}}{\reflection{\alpha_2}\reflection{\alpha_1}} &\defeq& 
    \BGG{\lambda}{\reflection{\alpha_1}}{\reflection{\alpha_2}\reflection{\alpha_1}}.
    \BGG{\lambda}{1}{\reflection{\alpha_1}}.
    \BGG{\lambda}{\reflection{\alpha_2}}{1} \\
  \BGG{\lambda}{\reflection{\alpha_1}\reflection{\alpha_2}}{\reflection{\alpha_1}} &\defeq& \BGG{\lambda}{\reflection{\alpha_1}}{\reflection{\alpha_1}\reflection{\alpha_2}}^* \\  
  \BGG{\lambda}{\reflection{\alpha_2}\reflection{\alpha_1}}{\reflection{\alpha_2}} &\defeq& \BGG{\lambda}{\reflection{\alpha_2}}{\reflection{\alpha_2}\reflection{\alpha_1}}^*.
\end{eqnarray*}
\end{definition}

Obviously, the definitions of the normalized BGG operators $\BGG{\lambda}{w}{w'}$ are identical to the definitions of the intertwining operators $\intertwiner{\lambda+\rho}{w}{w'}$, except that the weights of the principal series representations on which they act differ by the shift of $\rho$.  The next few lemmas describe the consequences of this.  To begin with, we have an exact analogue of Lemma \ref{lem:intertwiners_in_A}, with essentially identical proof.

\begin{lemma}
\label{lem:BGG_in_A}
Let $\lambda$ be a dominant weight.
\begin{enumerate}
\item  For each arrow $\edge{w}{w'}$, $\BGG{\lambda}{w}{w'} \in \scrA$.
\item  If $\edge[\alpha]{w}{w'}$, then $[\BGG{\lambda}{w}{w'}, \multop{f}] \in \scrK[\alpha]$ for any $f\in C(\scrX)$.
\end{enumerate}
\end{lemma}

\begin{lemma}
\label{lem:po1_definition}
Let $\sectionpoi_1,\ldots,\sectionpoi_k\in\CXE[\rho]$ be such that $\sum_{j=1}^k |\sectionpoi_j|^2 =1$, as in Lemma \ref{lem:partition_of_unity}.  If $\edge[\alpha]{w}{w'}$, then
$$
  \BGG{\lambda}{w}{w'}   
    \equiv \sum_{j=1}^k \multop{\overline{\sectionpoi_j}} \intertwiner{\lambda+\rho}{w}{w'} \multop{\sectionpoi_j} 
    \pmod{\scrK[\alpha]}. 
$$
\end{lemma}

\begin{proof}
This is an immediate consequence of Lemma \ref{lem:F_f_commute}.
\end{proof}

\begin{lemma}
\label{lem:F2-1_in_K}
If $\edge[\alpha]{w}{w'}$, then  $\BGG{\lambda}{w'}{w} \BGG{\lambda}{w}{w'}-1 \in \scrK[\alpha]$.
\end{lemma}

\begin{proof}
Let $\sectionpoi_1,\ldots,\sectionpoi_k\in\CXE[\rho]$ be as in the previous lemma.  By Lemmas \ref{lem:po1_definition} and \ref{lem:intertwiners_in_A},
\begin{eqnarray*}
  \BGG{\lambda}{w'}{w} \BGG{\lambda}{w}{w'} 
   &\equiv& \sum_{j,j'} \multop{\overline{\sectionpoi_j}} \intertwiner{\lambda+\rho}{w'}{w'} \multop{\sectionpoi_j \overline{\sectionpoi_{j'}}} \intertwiner{\lambda+\rho}{w}{w'} \multop{\sectionpoi_{j'}} \pmod{\scrK[\alpha]}\\
   &\equiv& \sum_{j,j'} \multop{\overline{\sectionpoi_j}} \intertwiner{\lambda+\rho}{w'}{w} \intertwiner{\lambda+\rho}{w}{w'}  \multop{\sectionpoi_j \overline{\sectionpoi_{j'}}}\multop{\sectionpoi_{j'}} \pmod{\scrK[\alpha]}\\
   &=& \sum_{j,j'} \multop{\overline{\sectionpoi_j} \sectionpoi_j \overline{\sectionpoi_{j'}} \sectionpoi_{j'}} \\
   &=& 1.
\end{eqnarray*}
\end{proof}


\begin{lemma}
\label{lem:diagram_commutes}
The diagram of normalized BGG operators
\begin{equation}
\label{eq:BGG_diagram}
  \xymatrix@!C=5ex{
    & \LXE[\reflection{\alpha_1}\shiftedaction\lambda] \ar@{<->}[ddrr] \ar@{<->}[rr] && \LXE[\reflection{\alpha_1}\reflection{\alpha_2}\shiftedaction\lambda] \ar@{<->}[dr] \\
   \LXE[\lambda] \ar@{<->}[ur] \ar@{<->}[dr] &&&& \LXE[w_\rho\shiftedaction\lambda] \\
    & \LXE[\reflection{\alpha_2}\shiftedaction\lambda] \ar@{<->}[uurr] \ar@{<->}[rr] && \LXE[\reflection{\alpha_2}\reflection{\alpha_1}\shiftedaction\lambda] \ar@{<->}[ur] 
  }
\end{equation}
commutes modulo $\scrK[\rho]$.
\end{lemma}

\begin{proof}
For adjacent edges $w\xleftrightarrow{\alpha} w' \xleftrightarrow{\alpha'} w''$, a calculation analogous to that of the previous proof gives
\begin{equation*}
  \BGG{\lambda}{w'}{w''} \BGG{\lambda}{w}{w'}
   \equiv \sum_{j,j'} \multop{\overline{\sectionpoi_j}} \left(\intertwiner{\lambda+\rho}{w'}{w''} \intertwiner{\lambda+\rho}{w}{w'}\right) \multop{\sectionpoi_{j} \overline{\sectionpoi_{j'}}}\multop{\sectionpoi_{j'}} 
   \pmod{\scrK[\alpha']}.
\end{equation*}
Note that $\scrK[\alpha']\subseteq\scrK[\rho]$.  The commutativity of \eref{eq:BGG_diagram} modulo $\scrK[\rho]$ is therefore a consequence of the commutativity of the corresponding diagram of intertwiners $\intertwiner{\lambda+\rho}{w}{w'}$ (Remark \ref{rmk:commuting_diagram_of_intertwiners}).
\end{proof}

\begin{lemma}
\label{lem:F_g_commute}
Let $\edge[\alpha]{w}{w'}$.  For any $g\in G$, 
\begin{equation}
\label{eq:F_g_commute}
  U_{w'\shiftedaction\lambda}(g) \BGG{\lambda}{w}{w'} U_{w\shiftedaction\lambda}(g^{-1}) - \BGG{\lambda}{w}{w'} \; \in \;
      \scrK[\alpha] 
\end{equation}
\end{lemma}

\begin{proof}
We first note that if $A$ is a $\Lie{G}$-continuous operator, then so is $g.A.g^{-1}$.
Let $\sectionpoi_1,\ldots,\sectionpoi_k\in\CXE[\rho]$ be as in Lemma \ref{lem:po1_definition}.  Then,
\begin{eqnarray}
\lefteqn{U_{w'\shiftedaction\lambda}(g) \BGG{\lambda}{w}{w'} U_{w\shiftedaction\lambda}(g^{-1})} \qquad \nonumber\\
 &\equiv& \sum_{j} U_{w'\shiftedaction\lambda}(g) \multop{\overline{\sectionpoi_j}} \intertwiner{\lambda+\rho}{w}{w'} \multop{\sectionpoi_j} U_{w\shiftedaction\lambda}(g^{-1}) \nonumber \pmod{\scrK[\alpha]}  \\
 &=& \sum_{j} U_{w'\shiftedaction\lambda}(g) \multop{\overline{\sectionpoi_j}}U_{w'(\lambda+\rho)}(g^{-1}) \intertwiner{\lambda+\rho}{w}{w'} U_{w(\lambda+\rho)}(g) \multop{\sectionpoi_j} U_{w\shiftedaction\lambda}(g^{-1}) \nonumber \\
 &=& \sum_{j} \multop{\overline{g\cdot \sectionpoi_j}} \intertwiner{\lambda+\rho}{w}{w'} \multop{g\cdot \sectionpoi_j}.
 \label{eq:F_g_commute_computation}
\end{eqnarray}
Since $\sum_{j=1}^k |g\cdot \sectionpoi_j|^2 =  1$, Lemma \ref{lem:po1_definition} shows that \eref{eq:F_g_commute_computation} equals $\BGG{\lambda}{w}{w'}$ modulo $\scrK[\alpha]$.

\end{proof}


\subsection{Construction of the gamma element}
\label{sec:gamma}

Fix a dominant weight $\lambda$.  Let $H_\lambda \defeq \bigoplus_{w\in\Weylgroup} \LXE[w\shiftedaction \lambda]$.  For each $w\in\Weylgroup$, let $\component{w}$ denote the orthogonal projection onto the summand $\LXE[w \shiftedaction \lambda]$ of $H_\lambda$.  We put a grading on $H_\lambda$ by declaring $\LXE[w\shiftedaction \lambda]$ to be even or odd according to the parity of $\length(w)$.  

For $f\in C(\scrX)$, $M_f$ will denote the multiplication operator on $H_\lambda$, acting diagonally on the summands.  We let  $U$ denote the diagonal representation $\oplus_{w\in\Weylgroup} U_{w\shiftedaction\lambda}$ of $\Lie{G}$.  For each $\edge{w}{w'}$, we extend the normalized BGG operator $\BGG{\lambda}{w}{w'}:\LXE[w\shiftedaction\lambda] \to \LXE[w'\shiftedaction\lambda]$ to an operator $\BGGextended{\lambda}{w}{w'}:H_\lambda \to H_\lambda$ by defining it to be zero on the components $\LXE[w''\shiftedaction \lambda]$ with $w''\neq w$. 

For the remainder of this section, we use $\scrK[\alpha]$, $\scrA$, $\scrK$, $\scrL$ to denote $\scrK[\alpha](H_\lambda)$, $\scrA (H_\lambda) $, $\scrK (H_\lambda) $, $\scrL (H_\lambda)$.

\begin{lemma}[Kasparov Technical Theorem]
\label{lem:KTT}
There exist positive $\Lie{G}$-continuous operators $N_1,N_2\in \scrL$ with the following properties:
\begin{enumerate}
\item $N_1^2+N_2^2 = 1$,
\item $N_i \cdot\scrK[\alpha_i] \subseteq \scrK$ for each $i=1,2$,
\item $N_i$ commutes modulo compact operators with
\begin{itemize}
  \item $\multop{f}$ for all $f\in C(\scrX)$,
  \item $U(g)$ for all $g\in\Lie{G}$,
  \item the normalized BGG operators $\BGGextended{\lambda}{w}{w'}$, for all $\edge{w}{w'}$,
\end{itemize}
\item $N_i$ commutes on the nose with $U(k)$ for all $k\in\Lie{K}$,
\item $N_i$ commutes on the nose with the projections $\component{w}$  for all $w\in\Weylgroup$, {\em i.e.}, $N_i$ is diagonal with respect to the direct sum decomposition of $H_\lambda$.
\end{enumerate}
\end{lemma}

Note also that  $N_1$ and $N_2$ commute, by \emph{(i)}.

\begin{proof}
See \cite[Theorem 20.1.5]{Blackadar}.  The $\Lie{K}$-invariance of (iv) is obtained by averaging over the $\Lie{K}$-translates $U(k)\, N_i\, U(k^{-1})$ of $N_i$.  Also, the operators $\sum_w \pm Q_w$ (taking all possible choices of signs) form a finite group of unitaries, so that a similar averaging trick gives property (v).
\end{proof}

\begin{lemma}
\label{lem:operator_po1}
There exist mutually commuting operators $\poI{w}{w'} \in \scrL$, indexed by the edges of the graph \eref{eq:Weyl_graph}, with the following properties:
\begin{enumerate}
\item $\poI{w}{w'} = \poI{w'}{w}$
\item If $\edge[\alpha]{w}{w'}$ for $\alpha\in\{\alpha_1,\alpha_2,\rho\}$, then  $\poI{w}{w'} \scrK[\alpha] \subseteq \scrK$.
\item If $w\leftrightarrow w' \leftrightarrow w''$ with $w\neq w''$ then $\poI{w'}{w''}\poI{w}{w'} \scrK[\rho] \subseteq \scrK$.
\item For any $w,w''\in\Weylgroup$, $\sum_{w'} \poI{w'}{w''}\poI{w}{w'} =\delta_{w,w''}$, where the sum is over $w'$ such that $w \leftrightarrow w' \leftrightarrow w''$.
\item $\poI{w}{w'}$ satisfies {(iii)}, {(iv)} and{(v)} of Lemma \ref{lem:KTT}.
\end{enumerate}
\end{lemma}

\begin{remark}
To clarify a possibly misleading notational point, $\poI{w}{w'}$ does not designate an operator between $\LXE[w\shiftedaction\lambda]$ and $\LXE[w'\shiftedaction\lambda]$.  Rather it is an operator on $H_\lambda$ which we will use to modify the operator $\BGG{\lambda}{w}{w'}$.
\end{remark}

\begin{proof}
With $N_1, N_2$ as in the previous lemma, assign operators $\poI{w}{w'}$ to each arrow as follows:
$$
  
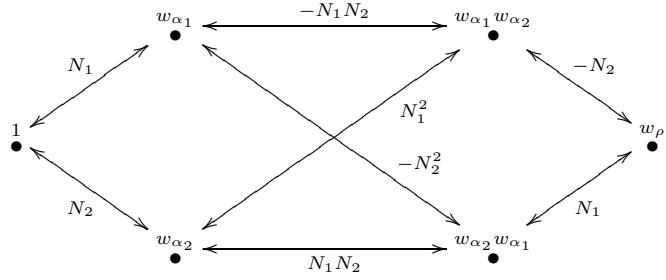
\xymatrix@!C{
    & \stackrel{\reflection{\alpha_1}}{\bullet} \ar@{<->}[rr]^{-N_1N_2} \ar@{<->}[ddrr]^(0.7){-N_2^2} 
      && \stackrel{\reflection{\alpha_1}\reflection{\alpha_2}}{\bullet} \ar@{<->}[dr]^{-N_2} \\
    \stackrel{1}{\bullet} \ar@{<->}[ur]^{N_1} \ar@{<->}[dr]_{N_2} 
      &&&& \stackrel{w_\rho}{\bullet} \\
    & \stackrel{\reflection{\alpha_2}}{\bullet} \ar@{<->}[rr]_{N_1N_2} \ar@{<->}[uurr]_(0.7){N_1^2} 
      && \stackrel{\reflection{\alpha_2}\reflection{\alpha_1}}{\bullet} \ar@{<->}[ur]_{N_1} 
  }
$$
The asserted properties can be easily checked using the properties of $N_1$ and $N_2$ from Lemma \ref{lem:KTT} and the diagram \eref{eq:Weyl_graph}.  It is worth noting particularly that $N_1N_2$ multiplies $\scrK[\rho]$ into the compact operators. 
\end{proof}

\begin{definition}
\label{def:Fredholm_components}
Define $\poIBGG{\lambda}{w}{w'} \defeq \poI{w}{w'} \BGGextended{\lambda}{w}{w'}$.
\end{definition}

\begin{lemma}
\label{lem:compact_commutators}
For any $\edge{w}{w'}$,
\begin{enumerate}
\item $\poIBGG{\lambda}{w}{w'} - \poIBGG{\lambda}{w'}{w}^*\in \scrK$.
\item $[\poIBGG{\lambda}{w}{w'}, M_f] \in \scrK$, for any $f\in C(\scrX)$,
\item $U(g)\,\poIBGG{\lambda}{w}{w'}\,U(g^{-1}) - \poIBGG{\lambda}{w}{w'} \in \scrK$, for any $g\in\Lie{G}$,
\item $\poIBGG{\lambda}{w}{w'}$ is $\Lie{K}$-invariant, {\em ie}, $[\poIBGG{\lambda}{w}{w'}, U(k)] =0$, for any $k\in\Lie{K}$,
\item $\poIBGG{\lambda}{w}{w'}$ is $\Lie{G}$-continuous.
\end{enumerate}
Also,
\begin{enumerate}
\setcounter{enumi}{5}
\item For any $w,w''\in\Weylgroup$, $\left( \sum_{w'} \poIBGG{\lambda}{w'}{w''} \poIBGG{\lambda}{w}{w'} \right) \equiv \delta_{w,w''} \component{w} \pmod{\scrK}$, where the sum is over $w'\in\Weylgroup$ such that $w\leftrightarrow w' \leftrightarrow w''$.

\end{enumerate}
\end{lemma}

\begin{proof}

Let $\edge[\alpha]{w}{w'}$. By definition, $\BGG{\lambda}{w}{w'} = \BGG{\lambda}{w'}{w}^*$, so $\poIBGG{\lambda}{w}{w'} - \poIBGG{\lambda}{w'}{w}^* = [\poI{w}{w'},\BGGextended{\lambda}{w}{w'}]$, which proves (i).  

Since $\poI{w}{w'}$ commutes modulo compacts with multiplication operators,
$$
  [ \poIBGG{\lambda}{w}{w'} , \multop{f}]
   \equiv \poI{w}{w'} [\BGGextended{\lambda}{w}{w'}, \multop{f}] \pmod{\scrK} 
$$
By Lemma \ref{lem:BGG_in_A}, the latter is in $\poI{w}{w'}\scrK[\alpha] \subseteq \scrK$, which proves (ii).  Similarly, for (iii),
\begin{multline*}
  U(g)\,\poIBGG{\lambda}{w}{w'}\,U(g^{-1}) - \poIBGG{\lambda}{w}{w'} \\
    \equiv \poI{w}{w'} \big( U(g)\,\BGGextended{\lambda}{w}{w'}\,U(g^{-1}) - \BGGextended{\lambda}{w}{w'} \big)\pmod{\scrK}
\end{multline*}
and the latter is in $\poI{w}{w'} \scrK[\alpha] \subseteq \scrK$ by Lemma \ref{lem:F_g_commute}.

For any weight $\mu$, the differential operator $X_i: \LXE[\mu]\to\LXE[\mu-\alpha_i]$ is $\Lie{K}$-invariant.  Likewise for its essential adjoint $Y_i:\LXE[\mu-\alpha_i] \to \LXE[\mu]$. Hence, $\Phase{X_i}:\LXE[\mu]\to\LXE[\mu-\alpha_i]$ is $\Lie{K}$-equivariant.  The normalized BGG operators $\BGG{\lambda}{w}{w'}$ are compositions of such operators, and $\poI{\lambda}{w}{w'}$ is $\Lie{K}$-invariant by definition.  This proves (iv).

Once again, $\Lie{G}$-continuity is trivial.

We prove (vi) in two separate cases.  Firstly, suppose $w=w''$.  For any $w'$ with $\edge{w}{w'}$, Lemma \ref{lem:F2-1_in_K} implies that $\BGGextended{\lambda}{w'}{w}\BGGextended{\lambda}{w}{w'} \equiv \component{w} \pmod{\scrK[\alpha]}$.  By Lemma \ref{lem:operator_po1}(iv),
\begin{eqnarray*}
  \sum_{w'} \poIBGG{\lambda}{w'}{w} \poIBGG{\lambda}{w}{w'}
    &\equiv& \sum_{w'} \poI{w'}{w}\poI{w}{w'} \BGGextended{\lambda}{w'}{w}\BGGextended{\lambda}{w}{w'}
      \pmod{\scrK} \\
    &\equiv& \sum_{w'} \poI{w'}{w}\poI{w}{w'} \component{w} \pmod{\scrK} \\
    &=& \component{w}.
\end{eqnarray*}
If $w\neq w'$, the result is trivial unless there exists at least one $w'$ such that $w \leftrightarrow w' \leftrightarrow w''$.  If such a $w'$ exists, Lemma \ref{lem:diagram_commutes} implies that the products $\BGG{\lambda}{w'}{w''} \BGG{\lambda}{w}{w'}$ are independent of this intermediate vertex $w'$, modulo $\scrK[\rho]$.  Let us fix one such product and denote it temporarily by $\BGG{\lambda}{w\to\cdot}{w''}$.  Then by Lemma \ref{lem:operator_po1}(iv),
\begin{eqnarray*}
  \sum_{w'} \poIBGG{\lambda}{w'}{w''} \poIBGG{\lambda}{w}{w'}
    &\equiv& \sum_{w'} \poI{w'}{w''}\poI{w}{w'} \BGGextended{\lambda}{w'}{w''}\BGGextended{\lambda}{w}{w'}
      \pmod{\scrK} \\
    &\equiv& \left(\sum_{w'} \poI{w'}{w}\poI{w}{w'}\right) \BGGextended{\lambda}{w\to \cdot}{w''}  \pmod{\scrK} \\
    &=& 0.
\end{eqnarray*}

\end{proof}

\begin{definition}
Define $\Fredholm{\lambda} \defeq\sum \poIBGG{\lambda}{w}{w'}$, where the sum is over all directed edges in the graph \eref{eq:Weyl_graph}.
\end{definition}

\begin{theorem}
\label{thm:main_theorem}
The operator $\Fredholm{\lambda}\in \scrL$ defines an element $\Kcycle{\lambda} \in K^\Lie{G}(C(\scrX),\CC)$.  That is, 
\begin{enumerate}
\item $\Fredholm{\lambda}$ is {\em odd} with respect to the grading of $H_\lambda$,
\item $\Fredholm{\lambda} - \Fredholm{\lambda}^* \in \scrK$,
\item $\Fredholm{\lambda}^2 - 1 \in \scrK$,
\item $[\Fredholm{\lambda}, M_f] \in \scrK$, for any $f\in C(\scrX)$,
\item $[\Fredholm{\lambda}, U(g)] \in \scrK$, for any $g\in\Lie{G}$,
\item $\Fredholm{\lambda}$ is $\Lie{G}$-continuous,
\end{enumerate}
Moreover, $\Fredholm{\lambda}$ is $\Lie{K}$-invariant: $[\Fredholm{\lambda},U(k)] = 0$ for all $k\in\Lie{K}$.
\end{theorem}

\begin{proof}
This is mostly immediate from the previous lemma.  To be explicit about the proof of (iii), Lemma \ref{lem:compact_commutators}(vi) gives
\begin{eqnarray*}
  \Fredholm{\lambda}^2 
    &=& \!\!\sum_{w,w',w''\in\Weylgroup \atop w\leftrightarrow w' \leftrightarrow w''} \!\!\poIBGG{\lambda}{w'}{w''} \poIBGG{\lambda}{w'}{w''} \\
    &\equiv& \sum_{w} \component{w}  \pmod{\scrK} \\
    &=& 1.
\end{eqnarray*}
\end{proof}

\begin{definition}
Let $\pi_\lambda$ denote the irreducible representation of $\Lie{K}$ with highest weight $\lambda$.  Define a homomorphism of abelian groups 
\begin{eqnarray*}
  \theta:R(\Lie{K}) &\to& \KK^\Lie{G}(C(\scrX),\CC) \\ 
     ~ [\pi_\lambda] &\mapsto& \Kcycle{\lambda}.
\end{eqnarray*}
\end{definition}

Let $\maptopt:\CC\to C(\scrX)$ denote the $G$-equivariant $C^*$-morphism induced by the map of $\scrX$ to a point.

\begin{theorem}
\label{thm:splitting}
The map $\maptopt^*\circ\theta:R(\Lie{K})\to  R(\Lie{G})$ is a ring homomorphism which splits the restriction homomorphism $\Res_\Lie{K}^\Lie{G}:R(\Lie{G}) \to R(\Lie{K})$.
\end{theorem}

\begin{proof}
Let $\lambda$ be a dominant weight.  We have that $\Res^\Lie{G}_\Lie{K}\maptopt^*\circ\theta([\pi_\lambda])$ is the $\Lie{K}$-index of $\Fredholm{\lambda}$.  Since $\Fredholm{\lambda}$ is $\Lie{K}$-equivariant, it decomposes as a direct sum of operators on the $\Lie{K}$-isotypical subspaces of $H_\lambda$, each of which is finite dimensional (Example \ref{ex:finite_multiplicities}).  The $\Lie{K}$-index of $\Fredholm{\lambda}$ is the sum of the indices of each component.

To compute this index, we compare with the classical BGG complex.  Let $\mu\defeq w\shiftedaction\lambda$ be in the shifted Weyl orbit of $\lambda$.  The induced bundle $E_\mu$ of our normalized $BGG$-complex and the holomorphic bundle $\Lhol{\mu}$ of the classical BGG complex \eref{eq:BGG_resolution} are identical as $\Lie{K}$-homogeneous line bundles.  The classical BGG resolution  is exact and $\Lie{K}$-equivariant, so exact in each $\Lie{K}$-type.  It follows that the index of $\Fredholm{\lambda}$ is $[\pi_\lambda]$.  Thus the composition $\Res_\Lie{K}^\Lie{G}\circ\maptopt^*\circ\theta$ is the identity on $R(\Lie{K})$.

By Theorem \ref{thm:split_surjection}, $\Res_\Lie{K}^\Lie{G}: \gamma R(\Lie{G}) \to R(\Lie{K})$ is a ring isomorphism, so it suffices to show that the image of $\maptopt^*\theta$ is in $\gamma R(\Lie{G})$.  Using \cite[Theorem 3.6(1)]{Kas88}, we have
\begin{eqnarray*}
\gamma \cdot (\maptopt^*\Kcycle{\lambda}) 
  &=& \maptopt^* \otimes_{C(\Lie{G/B})} (1_{C(\Lie{G/B})}\otimes\gamma)
    \otimes_{C(\Lie{G/B})} \Kcycle{\lambda} \\
  &=& \maptopt^* \otimes_{C(\Lie{G/B})} (\Ind_\Lie{B}^\Lie{G} \Res_\Lie{B}^\Lie{G}\gamma) \otimes_{C(\Lie{G/B})} \Kcycle{\lambda}. 
\end{eqnarray*}
Since $\Lie{B}$ is amenable, $\gamma$ restricts to the unit in $R(\Lie{B})$, so  $\gamma \cdot (\maptopt^* \Kcycle{\lambda}) = \maptopt^* \Kcycle{\lambda}$.

\end{proof}


\begin{corollary}
\label{cor:gamma}
 $\gamma = [(H_0, U, \Fredholm{0})] \in R(\Lie{G})$.
\end{corollary}


\appendix

\section{Harmonic analysis of longitudinal pseudodifferential operators}
\label{sec:PsiDOs_in_A}

This appendix describes the proof of Theorem \ref{thm:PsiDOs_in_A}(ii).
As in \cite{Yuncken:PsiDOs}, the key computation will be made using Gelfand-Tsetlin bases.    The following summary of Gelfand-Tsetlin bases follows the expository paper \cite{Molev} together with some remarks of \cite{Yuncken:PsiDOs}.  We immediately specialize to the case of $\slx(3,\CC)$.

Weights for $\glx(3,\CC)$ correspond to triples of integers $\mathbf{m} = (m_1,m_2,m_3)$ via
$$
  \mathbf{m}: \smatrix{t_1&0&0\\0&t_2&0\\0&0&t_3} \mapsto \sum_i m_it_i.
$$
Dominant weights correspond to descending triples, $m_1\geq m_2 \geq m_3$.

A {\em Gelfand-Tsetlin pattern} is an array of integers
$$
  \pattern \defeq \GTs{\lambda_{3,1}}{\lambda_{3,2}}{\lambda_{3,3}}{\lambda_{2,1}}{\lambda_{2,2}}{\lambda_{1,1}}
$$
satisfying the interleaving conditions
\begin{equation}
\label{eq:interleaving}
  \lambda_{k+1,j} \geq \lambda_{k,j} \geq \lambda_{k+1,j+1}.
\end{equation}
To each Gelfand-Tsetlin pattern there is associated a vector $\xi_\Lambda$ in the irreducible representation $\pi_\mathbf{m}$ with highest weight $\mathbf{m} = (\lambda_{31}, \lambda_{32}, \lambda_{33})$.  These vectors $\xi_\Lambda$ form an orthogonal (not orthonormal) basis for this representation.

When dealing with $\slx(3,\CC)$, rather than $\glx(3,\CC)$, there is some redundancy here.  Two triples describe the same $\slx(3,\CC)$-weight if and only if they differ by a multiple of $(1,1,1)$.  Two Gelfand-Tsetlin patterns describe the same basis vector if and only if they differ by a multiple of the constant pattern
$$
  \GTs111111.
$$

We use the following standard notation: $s_k \defeq \sum_{j=1}^k \lambda_{k,j}$ is the sum of the entries of the $k$th row; $l_{k,j} = \lambda_{k,j}-j+1$; and $\Lambda \pm \delta_{k,j}$ denotes the Gelfand-Tsetlin pattern obtained from $\Lambda$ by adding $\pm1$ to the $(k,j)$-entry.  Then
\begin{itemize}

\item  $\xi_\Lambda$ is a weight vector, with weight $(s_1-s_0, s_1-s_2, s_3-s_2)$.

\item  The representation $\pi_\mathbf{m}$ acts on this basis infinitesimally as follows:
\begin{eqnarray*}
  \pi(X_1) \xi_\Lambda
    &=& - (l_{11} - l_{21}) (l_{11} - l_{22})  \xi_{\Lambda+\delta_{11}} \\
  \pi(X_1^*) \xi_\Lambda
    &=&  \xi_{\Lambda-\delta_{11}} \\
  \pi(X_2) \xi_\Lambda
    &=&  -\frac{(l_{21}-l_{31})(l_{21}-l_{32})(l_{21}-l_{33})}{(l_{21}-l_{22})} \xi_{\Lambda+\delta_{21}} \\
    && \qquad
      -\frac{(l_{22}-l_{31})(l_{22}-l_{32})(l_{22}-l_{33})}{(l_{22}-l_{21})} \xi_{\Lambda+\delta_{22}} \\
  \pi(X_2^*) \xi_\Lambda
    &=&  \frac{(l_{21}-l_{11})}{(l_{21}-l_{22})} \xi_{\Lambda-\delta_{21}}
      +\frac{(l_{22}-l_{11})}{(l_{22}-l_{21})} \xi_{\Lambda-\delta_{22}}
\end{eqnarray*}

\item The norm of $\xi_\Lambda$ is given by
$$
  \| \xi_\Lambda \|^2 = \prod_{k=2}^3 
    \prod_{1\leq i \leq j < k} \frac{(l_{ki}-l_{k-1,j})!} {(l_{k-1,i}-l_{k-1,j})!}
    \prod_{1\leq i < j \leq k} \frac{(l_{ki}-l_{kj}-1)!} {(l_{k-1,i}-l_{k,j}-1)!}
$$

\item The vector $\xi_\Lambda$ lies in an irreducible representation for the Lie subalgebra
$$
  \lie{s}_1 =  \ulmatrix{\slx(2,\CC)}{0}{0}{0&0&0} \cong \slx(2,\CC)
$$
with highest weight is given by the second row $(\lambda_{21}, \lambda_{22})$ of $\Lambda$ (again, modulo multiples of $(1,1)$).
\end{itemize}

\begin{remark}
\label{rmk:K1-type_of_GTs_vectors}
As in Section \ref{sec:principal_series}, we note that the $\Lie{K}_1$-type of a Gelfand-Tsetlin vector $\xi_\Lambda$ in a given weight space is determined by the $\lie{s}_1$-type, hence by second row of $\Lambda$.

\end{remark}

Let us lift the longest element $\reflection{\rho} \in \Weylgroup$ to an element
\begin{equation}
\label{eq:long_Weyl_group_element}
  w_\rho \defeq \smatrix{0&0&-1 \\ 0&-1&0 \\ -1&0&0} \in \Lie{K}.
\end{equation}
Conjugation by $w_\rho$ interchanges the subgroups $\Lie{K}_1$ and $\Lie{K}_2$.  We define $\eta_\Lambda = \pi_\mathbf{m}(w_\rho) \xi_\Lambda$.  These vectors form an alternative orthogonal basis for $\pi_{\mathbf{m}}$ with related properties.  In particular, $\eta_\Lambda$ has weight 
$$
  w_\rho\cdot(s_1-s_0, s_1-s_2, s_3-s_2) = (s_3-s_2, s_2-s_1, s_1-s_0),
$$
norm $\|\eta_\Lambda\| = \|\xi_\Lambda\|$, and $\eta_\Lambda$ lies in an irreducible $\lie{s}_2$-subrepresentation with highest weight determined by the second row of $\Lambda$.  

We now compare the relative position of these two bases.  We begin with the representation with highest weight $\mathbf{m}=(m,0,-m)$ for $m\in\NN$.  Denote this representation by $\pi_m$.

The $0$-weight space of $V^{(m,0,-m)}$ is spanned by the Gelfand-Tsetlin vectors
$$
  \xi_{m,j} \defeq \xi_\Lambda, \qquad  \text{with } \Lambda = \GTs{m}0{-m}{j}{-j}0,
$$
for $j=0,\ldots, m$.  The $(0,-1,1)$-weight space is spanned by the vectors
$$
  \xi_{m,j}' \defeq \xi_\Lambda, \qquad  \text{with } \Lambda = \GTs{m~\;}0{\;-m}{\hspace{-1.5ex}(j\!-\!1)}{\!\!-j~}0,
$$
for $j=1,\ldots, m$.  By the Gelfand-Tsetlin formulas above,
\begin{eqnarray}
  \pi_m (X_2^*) \xi_{m,j} &=& \frac{j}{2j+1} \xi'_{m,j} + \frac{j+1}{2j+1} \xi'_{m,j+1}
    \label{eq:GTs_formula1} \\
  \pi_m (X_2) \xi_{m,j}' &=& \half ((m+1)^2 - j^2) \xi_{m,j-1} + \half ((m+1)^2 - j^2) \xi_{m,j}
    \label{eq:GTs_formula2}
\end{eqnarray}
so that
\begin{eqnarray}
  \pi_m (X_2)\pi_m (X_2^*) \xi_{m,j} 
    &=& \frac{j}{2(2j+1)}  ((m+1)^2-j^2) \xi_{m,j-1} \nonumber \\
    && \quad + \half ((m+1)^2 - (j^2+j+1)) \xi_{m,j} \nonumber \\
    && \quad + \frac{j+1}{2(2j+1)} ((m+1)^2-(j+1)^2) \xi_{m,j+1}.
      \label{eq:Laplacian2}
\end{eqnarray}
Also,
\begin{equation}
  \pi_m (X_1^*)\pi_m (X_1) \xi_{m,j} = j(j+1) \xi_{m,j}.
\label{eq:Laplacian1}
\end{equation}
The norms of these vectors are
\begin{eqnarray}
  \| \xi_{m,j} \|^2 &=& \frac{1}{2j+1} m!^2 (2m+1)!, 
    \label{eq:norm1} \\
  \| \xi_{m,j}' \|^2 &=& \frac{1}{2j} ((m+1)^2 - j^2)\, m!^2 (2m+1)!.
    \label{eq:norm2} 
\end{eqnarray}

We next define
\begin{eqnarray*}
  \eta_{m,j} &\defeq& \pi_m(w_\rho) \xi_{m,j}, \qquad (0\leq j \leq m) \\
  \eta_{m,j}' &\defeq& \pi_m(w_\rho) \xi_{m,j}', \qquad (1\leq j \leq m) 
\end{eqnarray*}
These vectors have weights $w_\rho\cdot0=0$ and $w_\rho\cdot(0,-1,1) = (1,-1,0)=\alpha_1$, respectively, and their norms $\|\eta_{m,j}\| = \|\xi_{m,j}\|$ and $\|\eta'_{m,j}\| = \|\xi'_{m,j}\|$ are given by Equations \eref{eq:norm1} and \eref{eq:norm2} above.  

For any $X\in\lie{k}_\CC$, $\pi_m(X) \eta_{m,j} = \pi_m(w) \pi_m(\ad(w)X) \xi_{m,j}$.  Since $\ad(w)X_1 = X_2^*$ and $\ad(w)X_1^* = X_2$, the formulae \eref{eq:GTs_formula1}--\eref{eq:Laplacian1} above give
\begin{eqnarray}
  \pi_m (X_1) \eta_{m,j} &=& \frac{j}{2j+1} \eta_{m,j} + \frac{j+1}{2j+1} \eta_{m,j+1},
    \label{eq:GTs_formula1a} \\
  \pi_m (X_1^*) \eta_{m,j}' &=& \half ((m+1)^2 - j^2) \eta_{m,j-1} + \half ((m+1)^2 - j^2) \eta_{j}, \nonumber\\
 ~    \label{eq:GTs_formula2a} \\
  \pi_m (X_1^*)\pi_m (X_1) \eta_{m,j} 
    &=& \frac{j}{2(2j+1)} ((m+1)^2-j^2) \eta_{m,j-1} \nonumber \\
    && \quad + \half ((m+1)^2 - (j^2+j+1)) \eta_{m,j} \nonumber \\
    && \quad + \frac{j+1}{2(2j+1)} ((m+1)^2-(j+1)^2) \eta_{m,j+1},
      \label{eq:Laplacian2a} \\
  \pi_m (X_2)\pi_m (X_2^*) \eta_{m,j} &=& j(j+1) \eta_{m,j}. \label{eq:Laplacian1a}
\end{eqnarray}

\begin{lemma}
\label{lem:K2-type_0}
For any $m\in \NN$,
  $$\eta_{m,0} = \omega_m \sum_{j=0}^m (-1)^j \frac{2j+1}{m+1} \xi_{m,j},$$
where $\omega_m\in\CC$ is some phase factor, $|\omega_m|=1$.
\end{lemma}

\begin{proof}

Write $\eta_{m,0} = \sum_{j=0}^m c_{m,j} \xi_{m,j}$.  Note that $\eta_{m,0}$ is annihilated by $\pi_m(X_2^*)$ (for instance, by Equation \eref{eq:Laplacian1a}), so Equation \eref{eq:GTs_formula1} gives, 
$$
  0 = \pi_m(X_2^*) \sum_{j=0}^m c_{m,j} \xi_{m,j} 
    \;=\; \sum_{j=0}^m c_{m,j} \left( \frac{j}{2j+1} \xi_{m,j}' + \frac{j}{2j+1} \xi_{m,j+1}' \right).
$$
Taking the coefficient of $\xi_{m,j}'$ in this equation gives
$c_{m,j} = -\frac{2j+1}{2j-1} c_{m,j-1}$.
By induction, $c_{m,j} = (-1)^j (2j+1) c_{m,0}$.  Hence
\begin{equation}
\label{eq:basis_comparison}
  \eta_{m,0} = c_{m,0} \sum_{j=0}^m (-1)^j (2j+1) \xi_{m,j}.
\end{equation}
Computing the norms of both sides of this using Equation \eref{eq:norm2}, we get
\begin{eqnarray*}
  m!^2 (2m+1)! &=& |c_{m,0}|^2 \sum_{j=0}^m (2j+1)^2 \frac{1}{(2j+1)} m!^2 (2m+1)! \\
    &=& |c_{m,0}|^2 (m+1)^2 m!^2 (2m+1)!
\end{eqnarray*}
Hence, $|c_{m,0}| = 1/(m+1)$, which completes the proof.
\end{proof}

Define 
\begin{equation}
\label{eq:a}
  a_{m,j,k} \defeq \frac{(-1)^j\overline{\omega_m}}{m!^2 (2m+1)!} \,\ip{ \xi_{m,j}, \eta_{m,k} }.
\end{equation}
Consistent with earlier convention, we put $a_{m,j,k} \defeq 0$ if $0\leq j,k \leq m$ does not hold.

\begin{lemma}
\label{lem:recurrence_relation}
For $0\leq j,k \leq m$, we have the recurrence relation in $k$
\begin{multline}
  k ((m+1)^2 - k^2) \,a_{m,j,k-1} \\
    \quad + (2k+1) ((m+1)^2 - (k^2+k+1) - 2j(j+1))\, a_{m,j,k} \\
    \quad + (k+1) ((m+1)^2 - (k+1)^2 )\, a_{m,j,k+1}  = 0,
\label{eq:recurrence_relation}
\end{multline}
with initial condition $a_{m,j,0} = \frac{1}{(m+1)}$.
\end{lemma}

\begin{remark}
When $k=0$, the first term in \eref{eq:recurrence_relation} vanishes, so that for each fixed $j$ and $m$, the one initial condition suffices to determine a solution for all $k$.
\end{remark}

\begin{proof}
Applying Equations \eref{eq:Laplacian1} and \eref{eq:Laplacian2a} to the equality
$$
 \ip{ \pi_m(X_1^*) \pi_m(X_1) \xi_{m,j} , \eta_{m,k} }  = \ip{ \xi_{m,j} ,  \pi_m(X_1^*) \pi_m(X_1) \eta_{m,k} },
$$
yields
\begin{eqnarray*}
   j(j+1) \ip{ \xi_{m,j} , \eta_{m,k} } 
  &=&   \frac{k}{2(2k+1)} ((m+1)^2-k^2) \ip{ \xi_{m,j}, \eta_{m,k-1} }\\
      && \quad + \half ((m+1)^2 - (k^2+k+1)) \ip{ \xi_{m,j}, \eta_{m,k} }\\
      && \quad + \frac{k+1}{2(2k+1)} ((m+1)^2-(k+1)^2) \ip{ \xi_{m,j}, \eta_{m,k+1} },
\end{eqnarray*}
which reduces to \eref{eq:recurrence_relation}.
For the initial condition, Lemma \ref{lem:K2-type_0} gives
$$
  a_{m,j,0} =  (-1)^j \frac{2j+1}{m+1} \frac{1}{m!^2 (2m+1)!} \| \xi_{m,j} \|^2 = \frac{(-1)^j}{(m+1)}.
$$

\end{proof}

We now give an approximate solution to the recurrence relation.

\begin{lemma}
\label{lem:approximate_solution}

Define 
\begin{equation}
\label{eq:b}
  b_{m,j,k} \defeq \textstyle \frac{1}{m+1} P_k \left(  2\left( \frac{j}{m+1} \right)^2 -1  \right),
\end{equation}
where $P_k$ is the $k$th Legendre polynomial.  
For each $k\in \NN$, there is a constant $C(k)$ independent of $j$ and $m$ such that
\begin{equation}
\label{eq:approximate_solution}
  | a_{m,j,k} - b_{m,j,k} | \leq C(k) (m+1)^{-2} \qquad \text{for all $m\geq k$, $0\leq j \leq m$}.
\end{equation}
\end{lemma}

\begin{proof}
Note that $b_{m,j,0} = \frac1{m+1} = a_{m,j,0}$, so that the case $k=0$ is trivial. 

Now fix $k\in\NN$, and assume inductively that $C(k)$ and $C(k-1)$ have been defined.  
Put $C_1(k+1) \defeq \max \{ (m+1)^2 |a_{m,j,k+1} - b_{m,j,k+1}| : k+1 \leq m < 2k, 0 \leq j \leq m \}$, so that
\begin{equation}
\label{eq:approximate_solution1}
  | a_{m,j,k+1} - b_{m,j,k+1} | \leq C_1(k) (m+1)^{-2} \quad \text{for all $0\leq j \leq m$ with $k+1 \leq m < 2k$}.
\end{equation}
Thus, we restrict to the case $m\geq 2k$.  

By a well known recurrence relation for Legendre polynomials,
$$
   (k+1)\, b_{m,j,k+1} - (2k+1) \left( \textstyle  2\left( \frac{j}{m+1} \right)^2 - 1 \right) \,b_{m,j,k}
    + k \,b_{m,j,k-1} = 0.
$$
Therefore,
\begin{multline*}
  (k+1) ((m+1)^2 - (k+1)^2 )\, b_{m,j,k+1} \\
    \quad + (2k+1) ((m+1)^2 - (k+1)^2) \left( \textstyle 1- 2\left( \frac{j}{m+1} \right)^2 \right) b_{m,j,k} \\
    \quad + k ((m+1)^2 - (k+1)^2)\, b_{m,j,k-1} = 0.
\end{multline*}
Subtracting this from Equation \eref{eq:recurrence_relation}, we get
\begin{multline}
  (k+1) ((m+1)^2 - (k+1)^2 ) (a_{m,j,k+1} - b_{m,j,k+1}) \\
    \quad + (2k+1) ((m+1)^2 - (k^2+k+1) - 2j(j+1)) (a_{m,j,k} - b_{m,j,k}) \\
    \quad + k ((m+1)^2 - k^2) (a_{m,j,k-1} - b_{m,j,k-1}) \\
    \quad \textstyle + (2k+1) \left(k-2j - 2(k+1)^2 \left( \frac{j}{m+1} \right)^2 \right) b_{m,j,k} \\
    + k(2k+1) b_{m,j,k-1} 
   = 0.
\label{eq:c_recurrence_relation}
\end{multline}

Note that $|P_k(x)| \leq 1$ for $x\in[-1,1]$.  Therefore $|b_{m,j,k}| \leq (m+1)^{-1}$ for $0\leq j\leq m$, and hence the latter two terms of Equation \eref{eq:c_recurrence_relation} are bounded by a constant $C_2(k)$ depending only on $k$.  Using the inductively assumed bounds on $|a_{m,j,k} - b_{m,j,k}|$ and $|a_{m,j,k-1} - b_{m,j,k-1}|$ we get
\begin{multline*}
  (k+1) ((m+1)^2 - (k+1)^2 ) |a_{m,j,k+1} - b_{m,j,k+1}| \\
    \leq (2k+1) 4 C(k)  + k \,C(k-1)  +  C_2(k).
\end{multline*}
If $m\geq 2k$, then 
$
 (k+1) ((m+1)^2 - (k+1)^2) > \half (k+1) (m+1)^2
$,
so that 
\begin{multline}
\label{eq:approximate_solution2}
  |a_{m,j,k+1} - b_{m,j,k+1}|   \leq [16 C(k) + 2 C(k-1) + 2 (k+1)^{-1} C_2(k+1)] (m+1)^{-2} \\
    \qquad \text{for all $0\leq j \leq m$ with $m\geq 2k$}.
\end{multline}
If $C(k+1)$ is the maximum of the constants $C_1(k+1)$ and $[16 C(k) + 2 C(k-1) + 2 (k+1)^{-1} C_2(k+1)]$, we are done.

\end{proof}

\begin{remark}
\label{rmk:a_bound}
We observed that $|b_{m,j,k}| \leq (m+1)^{-1}$ for all $0\leq j,k\leq m$.  In the light of the estimate \eref{eq:approximate_solution} we also have $|a_{m,j,k}| \leq (1+C(k)) (m+1)^{-1}$. 
\end{remark}

Now consider the action of $X_1\in \lie{k}_\CC$ on the zero weight space of $\pi_{(m,0,-m)}$.  Note that  $X_1$ maps $p_0 V^{\pi_{(m,0,-m)}}$ to $p_{\alpha_1} V^{\pi_{(m,0,-m)}}$.  From Equation \eref{eq:Laplacian1}, 
\begin{equation}
\label{eq:phase_X1}
  (\Phase{X_1}) \xi_{m,j} = X_1.(X_1^*X_1)^{-\half} \xi_{m,j} = \frac{1}{\sqrt{j(j+1)}} X_1 \xi_{m,j}.
\end{equation}
for $j>0$, and $(\Phase{X_1}) \xi_{m,0} = 0$.
We next give an approximate formula for $\Phase{X_1}
$ with respect to the alternative basis $\{\eta_\Lambda\}$.

Recall that $\{\eta_{m,j}\}_{j=0}^m$ and $\{\eta_{m,j}'\}_{j=1}^m$ are orthogonal bases for $p_0V^{(m,0,-m)}$ and $p_{\alpha_1}V^{(m,0,-m)}$, respectively.   We let 
\begin{eqnarray}
  \yvec_{m,j} &\defeq& \eta_{m,j} / \| \eta_{m,j} \| = \frac{1}{m! (2m+1)!^\half} \eta_{m,j}, 
    \label{eq:yvec1} \\
  \yvec_{m,j}' &\defeq& \eta_{m,j}' / \| \eta_{m,j}' \| = \frac{1}{m! (2m+1)!^\half} 
    \sqrt{\frac{2j}{(m+1)^2-j^2}} \eta_{m,j} 
          \label{eq:yvec2} 
\end{eqnarray}
be the corresponding orthonormal bases.

\begin{lemma}
\label{lem:limit}

For each fixed $k\in\NN$, 
$$
  \lim_{m \to \infty} | \ip{ (\Phase{X_i}) \yvec_{m,0}, \yvec_{m,k}' } | = \sqrt{2k} \left(\frac{1}{2k-1} - \frac{1}{2k+1} \right).
$$

\end{lemma}

\begin{proof}

Using, successively, Equations \eref{eq:yvec1} and \eref{eq:yvec2}, Lemma \ref{lem:K2-type_0}, Equation \eref{eq:phase_X1}, Equation \eref{eq:GTs_formula2a}, and Equation \eref{eq:a}, we compute
\begin{eqnarray}
\lefteqn{   \ip{ (\Phase{X_1}) \yvec_{m,0}, \yvec_{m,k}' } }  \nonumber \\
    &=& \frac{1}{m!^2(2m+1)!} \sqrt{\frac{2k}{(m+1)^2-k^2}} \ip{ (\Phase{X_1}) \eta_{m,0}, \eta_{m,k}' } \nonumber\\
    &=& \frac{\omega_m}{m!^2(2m+1)!} \sqrt{\frac{2k}{(m+1)^2-k^2}} 
       \sum_{j=0}^m (-1)^j \frac{2j+1}{m+1} \ip{ (\Phase{X_1}) \xi_{m,j}, \eta_{m,k}' } \nonumber\\
     &=& \frac{\omega_m}{m!^2(2m+1)!} \sqrt{\frac{2k}{(m+1)^2-k^2}} 
       \sum_{j=0}^m   \frac{(-1)^j}{m+1} \frac{2j+1}{\sqrt{j(j+1)}} \ip{ (X_1 \xi_{m,j}, \eta_{m,k}' } 
         \nonumber \\
     &=& \frac{\omega_m}{m!^2(2m+1)!} \frac{1}{(m+1)} \sqrt{\frac{2k}{(m+1)^2-k^2}} 
       \sum_{j=0}^m  \frac{(-1)^j (2j+1)}{\sqrt{j(j+1)}} \ip{ (\xi_{m,j}, X_1^* \eta_{m,k}' } \nonumber\\
     &=& \frac{\omega_m}{m!^2(2m+1)!} \frac{\sqrt{2k ((m+1)^2 - k^2)} }{(m+1)} 
          \nonumber\\
     && \hspace{3cm} \times     \sum_{j=0}^m  (-1)^j \frac{2j+1}{2\sqrt{j(j+1)}} \ip{ (\xi_{m,j},  \eta_{m,k-1} + \eta_{m,k} } 
         \nonumber \\
     &=& \omega_m \sqrt{2k \left( 1 - \frac{k^2}{(m+1)^2} \right)}
       \sum_{j=0}^m  \frac{j+\half}{\sqrt{j(j+1)}} (a_{m,j,k-1} + a_{m,j,k} ) .
     \label{eq:F_computation1}  
\end{eqnarray}

Write the in the final line as
\begin{eqnarray}
\lefteqn{ \sum_{j=0}^m  \frac{j+\half}{\sqrt{j(j+1)}} (a_{m,j,k-1} + a_{m,j,k} ) } 
    \label{eq:sum} \\
  &=& \sum_{j=0}^m  \left( \frac{j+\half}{\sqrt{j(j+1)}} -1 \right) (a_{m,j,k-1} + a_{m,j,k} )
    \label{eq:error1} \\
  && + \sum_{j=0}^m  (a_{m,j,k-1}-b_{m,j,k-1}) + (a_{m,j,k} - b_{m,j,k} ) 
    \label{eq:error2} \\
  && + \sum_{j=0}^m (b_{m,j,k-1} + b_{m,j,k}).
    \label{eq:limit}
\end{eqnarray}
In the sum \eref{eq:error1}, 
$$
  \left( \frac{j+\half}{\sqrt{j(j+1)}} -1 \right) =
  \sqrt{ \frac{j^2+j+\frac{1}{4}}{j^2+j}} -1  
    \leq \frac{1}{8j^2},
$$  
which is summable, while $a_{m,j,k-1}$ and $a_{m,j,k}$ are both $O(m^{-1})$ for fixed $k$, by Remark \ref{rmk:a_bound}, so \eref{eq:error1} tends to $0$ as $m\to\infty$.  By Lemma \ref{lem:approximate_solution}, the sum \eref{eq:error2} is bounded by $(m+1). [C(k-1) + C(k)] (m+1)^{-2}$, so also vanishes as $m\to \infty$.  

Finally,
\begin{equation}
\sum_{j=0}^m b_{m,j,k} 
  = \frac{1}{m+1} \sum_{j=0}^{m} \textstyle  P_k\left( 2\left(\frac{j}{m+1} \right)^2 - 1\right)
\label{eq:Riemann_sum}
\end{equation}
is a Riemann sum for the integral $(-1)^j \int_0^1 P_k (2t^2 -1) \, dt$ 
Thus, using the substitution $u=1-2t^2$, \eref{eq:Riemann_sum} converges to
$$
  2^{-\frac{3}{2}}  \int_{-1}^1 (1-u)^{-\half} P_k(-u) \, du = \frac{(-1)^k}{2k+1}
$$
(see {\em e.g.,} \cite[7.225(3)]{Table-of-integrals}).  We obtain that the sum \eref{eq:limit}, and hence \eref{eq:sum} converges to $(-1)^{k-1}\left(\frac{1}{2k-1} - \frac{1}{2k+1}\right)$ as $m\to\infty$.  Putting this into Equation \eref{eq:F_computation1} proves the result.

\end{proof}

In the following lemmas, $\triv$ will denote the trivial representation of $\Lie{K}_2$.

\begin{lemma}
\label{lem:F_in_A_estimate}
On any unitary $\Lie{K}$-representation $\scrH$, $(\Phase{X_1}) p_\triv \in \scrK[\alpha_2](\scrH)$.

\end{lemma}

\begin{proof}
We will prove the equivalent condition of Proposition \ref{prop:KS_equivalent_conditions}(ii).  Note that if $F\subset\Khat_2$ contains $\sigma_0$ then $(\Phase{X_1}) p_\triv p_\irrepset^\perp =0$, so we only need that
for any $\epsilon>0$ there is a finite set  $\irrepset \subset \Khat_2$ such that 
\begin{equation}
\label{eq:F_in_A_estimate}
  \| p_\irrepset^\perp (\Phase{X_1}) p_\triv \| < \epsilon.
\end{equation}
We begin with the case of $\scrH$ an irreducible $\Lie{K}$-representation, say of highest weight $(m_1,m_2,m_3)$.

If $\eta\in\scrH$ is of trivial $\Lie{K}_2$-type, then it has trivial $\lie{s}_2$-type and weight $0$.  By the defining properties of the Gelfand-Tsetlin basis, $\eta$ must be a scalar multiple of $\eta_{m,0}$ for some $m\in\NN$.
Therefore $p_\triv =0$ on any irreducible representation other than the representations $\pi_{(m,0,-m)}$ considered above.

On $V^{\pi_{(m,0,-m)}}$, $p_\triv$ is the projection onto the span of $\yvec_{m,0}$.  Let $\sigma_k'$ denote the $\Lie{K}_2$-type in the $\alpha_1$-weight space with highest weight $(k-1,-k)$ for $\lie{s}_2$, so that $p_{\sigma_k'}$ is the projection onto $\yvec_{m,k}'$.
Then,
$$
  \|p_{\sigma_k}' (\Phase{X_1}) p_\triv \|^2 
    = |  \ip{ \yvec_{m,k}' , (\Phase{X_1}) \yvec_{m,0} } |^2 .
$$

Choose $l$ with $(2l+1)^{-2} < \half \epsilon$.  Applying Lemma \ref{lem:limit} to $k=1,\ldots,l$, we can find $M$ sufficiently large that for all $k \leq l$, 
$$
  | \ip{ \yvec_{m,k}' , (\Phase{X_1}) \yvec_{m,0} } |^2
    \geq 2k \left( \frac{1}{2k-1} - \frac{1}{2k+1} \right)^2 - \frac{1}{2l} \epsilon
      \qquad \text{for all $m \geq M$}.
$$
Putting $F_l \defeq \{\sigma_1',\ldots, \sigma_l'\}$, we get
\begin{eqnarray*}
  \| p_{F_l}^\perp (\Phase{X_1}) p_\triv \|^2
    &=& \|(\Phase{X_1}) p_\triv \|^2 
      - \sum_{k=1}^l \| p_{\sigma_k'} (\Phase{X_1}) p_\triv \|^2 \\
    &\leq& 1 - \sum_{k=1}^l \left( 2k \left( \frac{1}{2k-1} - \frac{1}{2k+1} \right)^2 - \frac{1}{2l} \epsilon \right) \\
    &\leq& 1 + \half \epsilon + \sum_{k=1}^l \frac{1}{(2k-1)^2} - \frac{1}{(2k+1)^2} \\
    &=& 1+ \half\epsilon - \left(1 - \frac{1}{(2l+1)^2} \right) \\
    &<& \epsilon,
 \label{eq:orthogonal_reduction*}
\end{eqnarray*}
so that \eref{eq:F_in_A_estimate} holds on all $\pi_{(m,0,-m)}$ with $m\geq M$.

In the finitely many $\Lie{K}$-representations $\pi_{(m,0,-m)}$ with $m< M$, there can only appear finitely many $\Lie{K}_2$-types.  Let $\irrepset\subset\Khat_2$ be the finite set containing all of these $\Lie{K}_2$-types as well as $\sigma_1'\,\ldots,\sigma_l'$ and $\sigma_0$.  Then $p_\irrepset^\perp (\Phase{X_1}) p_\triv  = 0$ on $V^{(m,0,-m)}$ for all $m<M$.  Hence $\|p_\irrepset^\perp (\Phase{X_1}) p_\triv\|<\epsilon$ on all irreducible representations.

For a general unitary $\Lie{K}$-representation $\scrH$, $p_\irrepset^\perp (\Phase{X_1}) p_\triv$ decomposes as a sum of operators on each irreducible component, so the same estimate holds.

\end{proof}

\begin{lemma}
\label{lem:F_in_A_estimate2}
On any unitary $\Lie{K}$-representation $\scrH$ the operators $(\Phase{X_1}^*) p_\triv$, and therefore $p_\triv (\Phase{X_1})$, are in $\scrK[\alpha_2](\scrH)$ .
\end{lemma}

\begin{proof}
Let $U$ be a unitary representation of $\Lie{K}$ on $\scrH$.  The antilinear map $J: \scrH \to \scrH^\dual;\; \xi \mapsto \ip{\xi,\slot}$ intertwines the representations $U$ and $U^\dual$.  One can check that for any $X$ in the complexification $\lie{k}_\CC$,  $J^{-1}U^\dual(X)J = -U(X)^*$.  Since $J$ is anti-unitary, $J^{-1}\Phase{(U^\dual(X))}J = -\Phase{(U(X^*))}$

If $\xi\in\scrH$ has $\Lie{K}_2$-type $\sigma$, then $J\xi$ has $\Lie{K}_2$-type $\sigma^\dual$,  so $J^{-1} p_\sigma^\dual J = p_\sigma$.  Thus, by conjugating by $J$, the estimate \ref{eq:F_in_A_estimate} implies $\| p_{\irrepset^\dual}^\perp (\Phase{X_1}^*) p_\triv \| < \epsilon$,   where $\irrepset^\dual \defeq \{ \sigma^\dual \st \sigma \in \irrepset\}$.

\end{proof}

This completes a base case in the proof of $\Phase{X_1}\in\scrA[\alpha_2]$. 
We now need to replace $\sigma_0$ by an arbitrary $\Lie{K}_2$-type in the preceding two lemmas.  To do so, we use a trick based on the following fact.  

\begin{lemma}
\label{lem:finitely_generated}
For any $\sigma\in\Khat_2$, there exists a finite collection of continuous functions $\psi_1,\ldots,\psi_n \in C(\Lie{K})$ such that $p_\sigma = \sum_{j=1}^n \multop{\psi_j} p_\triv \multop{\overline{\psi_j}}$ as an operator on $L^2(\Lie{K})$.
\end{lemma}


\begin{proof}
By the Peter-Weyl Theorem, the right regular representation of $\Lie{K}_2$ has a finite dimensional $\sigma$-isotypical subspace with a basis consisting of continuous functions $b_1,\ldots,b_m\in C(\Lie{K}_2)$.  Thus, for any $f\in L^2(\Lie{K}_2)$, 
\begin{equation}
\label{eq:K2-basis}
  p_\sigma f = \sum_{j=1}^n b_j \ip{b_j,f} = \sum_{j=1}^n \multop{b_j} p_\triv \multop{\overline{b_j}}f.
\end{equation}

Let $\scr{Y}\subset\Lie{K}/\Lie{K}_2$ be open and $\zeta: \scr{Y} \to \Lie{K}$ be a continuous local section of the principal $\Lie{K}_2$-bundle $q:\Lie{K}\to\Lie{K}/\Lie{K}_2$.  We now define functions on $\Lie{K}$ which equal $b_1,\ldots,b_m$ on each fibre over $\scr{Y}$.  To be precise, define $b^\scr{Y}_1,\ldots,b^\scr{Y}_m$ on $\Lie{K}$ by
$$
  b^\scr{Y}_j(k) \defeq \begin{cases}
   b_j(h), \qquad &\text{if $k=\zeta(y)h$ for some $y\in \scr{Y}$, $h\in\Lie{K}_2$},\\
   0 &\text{if $q(k) \notin \scr{Y}$}.
  \end{cases}
$$
By applying \eref{eq:K2-basis} fibrewise we get $p_\sigma f =  \sum_{j=1}^n \multop{b^\scr{Y}_j} p_\triv \multop{\overline{b^\scr{Y}_j}}f$ for any $f \in L^2(K)$ supported on $q^{-1}(\scr{Y})$. 

Now let $\scr{U} = \{(\scr{Y}_l, \zeta_l)\}$ be a finite atlas of such gauges.  Let $a_l \in C(\Lie{K}/\Lie{K}_2)$ be such that $\{a_l^2\}$ is a partition of unity subordinate to this atlas, and pull back to $\tilde{a}_l \defeq a_l\circ q \in C(\Lie{K})$.  Then $b^{\scr{Y}_l}_j\tilde{a}_l$ is continuous on $\Lie{K}$ for each $j,l$, and
for any $f\in L^2(K)$,
$$
  p_\sigma f = \sum_l a_l^2 f = \sum_l \sum_j \multop{(b^{\scr{Y}_l}_j\tilde{a}_l)} p_\triv \multop{\overline{(b^{\scr{Y}_l}_j\tilde{a_l})} }f.
$$
\end{proof}

\begin{lemma}
\label{lem:phase_multops_commute}
Let $\nu$ be a weight of $\Lie{K}$.  For any $f \in C(\Lie{K})$, $[\Phase{X}_1, \multop{f}]p_\nu$ and $ [\Phase{Y}_1, \multop{f}]p_\nu$ are in $\scrK[\alpha_1](L^2(\Lie{K}))$.
\end{lemma}

\begin{proof}
Suppose first that $f$ is a weight vector for the right regular representation, {\em i.e.}, $f\in\CXE[-\mu]$ for some $\mu$. Then Lemma \ref{lem:F_f_commute} says that
$$
  [\Phase{X}_1, \multop{f}]: p_\nu L^2(\Lie{K}) \to p_{\nu+\mu+\alpha_1} L^2(\Lie{K})
$$
is in $\scrK[\alpha_1]$, which implies the result.  The subspace spanned by these weight vectors contains all matrix units, so is uniformly dense in $C(\Lie{K})$.  A density argument completes the proof.  Similarly, $ [\Phase{Y}_1, \multop{f}]p_\nu \in \scrK[\alpha_1]$.
\end{proof}

\begin{theorem}
\label{thm:phase_in_A}

On any unitary $\Lie{K}$-representation $\scrH$, $\Phase{X_i}$ and $\Phase{Y_i}$ are in $ \scrA(\scrH)$ for $i=1,2$.

\end{theorem}

\begin{proof}
We begin with $\scrH=L^2(\Lie{K})$ with the right regular representation, and consider $\Phase{X_1}$.  As in Lemma \ref{lem:degeneracy}, the  finite multiplicity of $\Lie{K}$-types in $L^2(\Lie{K})$ implies that $\scrA[\Sigma](L^2(\Lie{K})) = \scrL(L^2(\Lie{K}))$, so $\Phase{X_1} \in \scrA[\Sigma]$ trivially.  Since $\Phase{X_1}$ maps the $\mu$-weight space into the $(\mu+\alpha_1)$-weight space for each weight $\mu$, it is $\Lie{M}$-harmonically proper, so in $\scrA[\emptyset]$.  Since $X_1\in(\lie{k_1})_\CC$, $\Phase{X_1}$ preserves $\Lie{K}_1$-types, so is in $\scrA[\alpha_1]$.  It remains to show $\Phase{X_1}\in\scrA[\alpha_2]$.

Let $\sigma\in\Khat_2$ and let $\psi_1,\ldots, \psi_n\in C(\Lie{K})$ be as in Lemma \ref{lem:finitely_generated}.  Then
\begin{eqnarray*}
  (\Phase{X_1}) p_\sigma
    &=& \sum_{j=1}^n  (\Phase{X_1}) \multop{\psi_j} p_\triv \multop{\overline{\psi_j}}  \\
    &=& \sum_{j=1}^n  \multop{\psi_j} (\Phase{X_1}) p_\triv \multop{\overline{\psi_j}} 
      + \sum_{j=1}^n   [(\Phase{X_1}), \multop{\psi_j}] p_\triv \multop{\overline{\psi_j}}. 
\end{eqnarray*}
Since $p_\triv$ projects into the $0$-weight space, Lemmas \ref{lem:F_in_A_estimate}, \ref{lem:phase_multops_commute} and \ref{lem:mult_ops_in_A}, imply  $(\Phase{X_1})p_\sigma \in \scrK[\alpha_2]$. 
A similar computation using Lemma \ref{lem:F_in_A_estimate2} shows that $(\Phase{Y_1})p_\sigma = (\Phase{X_1}^*)p_\sigma \in \scrK[\alpha_2]$, so  $p_\sigma(\Phase{X_1}) \in \scrK[\alpha_2]$. By Proposition \ref{prop:AS_equivalent_conditions}, $\Phase{X_1} \in \scrA[\alpha_2]$.  

Therefore $\Phase{X_1}\in\scrA$.  By taking adjoints, $\Phase{Y_1}\in \scrA$.  

Conjugation by the longest Weyl group element interchanges $\scrA[\alpha_1]$ and $\scrA[\alpha_2]$ and fixes $\scrA[\emptyset]$ and $\scrA[\Sigma]$, so fixes $\scrA$.  It also sends $X_1$ and $Y_1$ to $Y_2$ and $X_2$, respectively.  We obtain $\Phase{Y_2}$, $\Phase{X_2} \in \scrA$.

The theorem remains true if $\scrH$ is a direct sum of arbitrarily many copies of the regular representation.  Since every unitary $\Lie{K}$-representation can be equivariantly embedded into such a direct sum, we are done.

\end{proof}

\begin{corollary}
\label{cor:F_in_A}

Let $\scrH$ be any unitary $\Lie{K}$-representation.
For $i=1,2$ and any weight $\mu$, 
$\Phase{X_i}: p_\mu\scrH \to p_{\mu+\alpha_i}\scrH$ is in $\scrA$.  

In particular, $\Phase{X_i}: \LXE[-\mu] \to \LXE[-\mu-\alpha_i] \in \scrA$.

\end{corollary}

The above result is sufficient for the applications of this paper.  But the generalization to arbitrary order zero longitudinal pseudodifferential operators (Theorem \ref{thm:PsiDOs_in_A}) is easily deduced from it, and perhaps useful for future applications.

\begin{proof}[Proof of Theorem \ref{thm:PsiDOs_in_A}]

Start with the case $E=E'= E_0$, the trivial line bundle over $\scrX$.  Recall Connes' short exact sequence 
$$
  \xymatrix@!R{
  0 \ar[r] &
  \overline{\PsiDO[i]^{-1}}(E_0) \ar[r] &
  \overline{\PsiDO[i]^{0}}(E_0) \ar[r]^{\Symbol[i]} &
  C(\cosphere\foliation[i]) \ar[r] &
  0.
  }
$$
We know from \cite{Yuncken:PsiDOs} that $\overline{\PsiDO[i]^{-1}}(E_0) \subseteq \scrA$.  Let $\scr{C} \subseteq C(\cosphere\foliation[i])$ be the image of  $\overline{\PsiDO[i]^{0}}(E_0) \cap \scrA$ under the longitudinal principal symbol map.  We prove that $\scr{C} = C(\cosphere\foliation[i])$ by showing that it separates the points of $\cosphere\foliation[i]$, in the sense of the Stone-Weierstrass Theorem.

For any $f\in C(\scrX)$, the multiplication operator $M_f$ is in $\PsiDO[i]^{0}(E_0) \cap \scrA$, so the function algebra $\scr{C}$ separates points in different fibres of $\cosphere\foliation[i]$.  The longitudinal principal symbol of $\Phase{X_i}$ separates points in the fibre at the identity coset (see Lemma \ref {lem:F_in_PsiDO}).  Let $\varphi \in\CXE[\alpha]$ be any smooth section of $E_\alpha$ which is nonzero at the identity coset. Then $M_\varphi \Phase{X_i} \in \PsiDO[i]^0(E_0) \cap \scrA$ and its principal symbol separates points of the fibre at the identity coset.  Conjugating by translations by $k\in\Lie{K}$, $\scr{C}$ separates points in any fibre.

Now suppose $E = E_\mu$, $E'=E_\nu$ are general $\Lie{K}$-homogeneous line bundles.  Find partitions of unity $\sectionpoi_1,\ldots,\sectionpoi_n\in\CXE[\mu]$, $\sectionpoi'_1,\ldots, \sectionpoi'_m\in \CXE[\nu]$ in the sense of Lemma \ref{lem:partition_of_unity}.  If $A\in\overline{\PsiDO[i]^0}(E,E')$, then $ M_{\overline{\sectionpoi'_j}} A M_{\sectionpoi_k} \in \overline{\PsiDO[i]^0}(E_0) \subseteq \scrA$ for each $j,k$.  Hence $A = \sum_{j,k} M_{\sectionpoi'_j} M_{\overline{\sectionpoi'_j}} A M_{\sectionpoi_k} M_{\overline{\sectionpoi_k}} \in \scrA$.

The case of higher dimensional bundles reduces to the above by decomposing equivariantly into line bundles.

\end{proof}

\bibliographystyle{alpha}

\bibliography{sl3-gamma}

\end{document}